\documentclass[11pt]{article}

\usepackage{amsfonts}
\usepackage{latexsym}
\usepackage{amsmath}
\usepackage{amssymb}
\usepackage{amsthm, xcolor}
\newtheorem{theorem}{Theorem}[section]
\newtheorem{lemma}[theorem]{Lemma}

\newtheorem{proposition}[theorem]{Proposition}
\newtheorem{corollary}[theorem]{Corollary}

\theoremstyle{definition}
\newtheorem{definition}[theorem]{Definition}
\newtheorem{example}{Example}

\theoremstyle{remark}
\newtheorem*{remark}{Remark}

\newtheorem*{note*}{Note}



\newcommand{\abs}[1]{\left\lvert#1\right\rvert}
\newcommand{\norm}[1]{\left\lVert#1\right\rVert}

\newcommand{\wtd}[1]{\widetilde{#1}}

\newtheorem{note}{Remark}
\def\note#1{\ifvmode\leavevmode\fi\vadjust{\vbox to0pt{\vss
 \hbox to 0pt{\hskip\hsize\hskip1em
\vbox{\hsize2cm\small\raggedright\pretolerance10000
 \noindent #1\hfill}\hss}\vbox to8pt{\vfil}\vss}}}
 
\begin{document}

\small

\title{Affine isoperimetric inequalities on flag
  manifolds}

\author{Susanna Dann\thanks{Thanks the Oberwolfach Research Institute
    for Mathematics for its hospitality and support, where part of
    this work was carried out.} \and Grigoris Paouris
  \thanks{Supported by Simons Foundation Collaboration Grant \#527498
    and NSF grant DMS-1812240.} \and Peter Pivovarov \thanks{Supported
    by NSF grant DMS-1612936.}}


\maketitle

\begin{abstract}
  Building on work of Furstenberg and Tzkoni, we introduce ${\bf
    r}$-flag affine quermassintegrals and their dual versions. These
  quantities generalize affine and dual affine quermassintegrals as
  averages on flag manifolds (where the Grassmannian can be considered
  as a special case). We establish affine and linear invariance
  properties and extend fundamental results to this new setting. In
  particular, we prove several affine isoperimetric inequalities from
  convex geometry and their approximate reverse forms. We also
  introduce functional forms of these quantities and establish
  corresponding inequalities.
\end{abstract}


\section{Introduction}

Affine isoperimetric inequalities provide a rich foundation for
understanding principles in geometry and analysis that arise in the
presence of symmetries. Among the most fundamental examples is the
Blaschke-Santal\'{o} inequality \cite{Sant} on the product of volumes
of an origin-symmetric convex body $L$ in $\mathbb{R}^n$ and its polar
$L^{\circ}=\{x\in \mathbb{R}^n:\langle x, y\rangle \leq 1, \;\forall
y\in L\}$. The latter asserts that this product is maximized for
ellipsoids, i.e.,
\begin{equation}
  \label{eqn:BS}
  \abs{L}\abs{L^{\circ}} \leq \omega_n^2,
\end{equation}
where $\omega_n$ is the volume of the unit Euclidean ball $B_2^n$.
The Blaschke-Santal\'{o} inequality, and its version for
non-origin-symmetric bodies, is one of several equivalent forms of the
affine isoperimetric inequality; see e.g., the survey
\cite{Lutwak_survey}. Moreover, it admits numerous extensions: for
example, $L_p$ versions \cite{LZ}, generalizations from convex bodies
to functions, e.g., \cite{Ball1}, \cite{AKM}, \cite{FM} with
applications to concentration of measure \cite{AKM}, \cite{Lehec08};
further functional affine isoperimetric inequalities, e.g.,
\cite{AKSW}; stronger versions in which stochastic dominance holds
\cite{CFPP}.

Another fundamental affine isoperimetric inequality is the Petty polar
projection inequality \cite{Petty}.  This concerns projection bodies,
which are special zonoids that play a fundamental role in convex
geometry, functional analysis, among other fields, e.g.,
\cite{S}, \cite{GarB}. The projection body of a convex body
$L\subseteq\mathbb{R}^n$ is the convex body $\Pi L$ defined by its
support function in direction $\theta\in S^{n-1}$ by $h_{\Pi
  L}(\theta)=\abs{P_{\theta^{\perp}}L}$, where $P_{\theta^{\perp}}$ is
the orthogonal projection onto $\theta^{\perp}$. The Petty projection
inequality asserts that the affine-invariant quantity
$\abs{L}^{n-1}\abs{(\Pi L)^{\circ}}$ is maximized by ellipsoids, i.e.,
\begin{equation}
  \label{eqn:Petty}
  \abs{L}^{n-1}\abs{(\Pi L)^{\circ}}\leq
  \omega_n^{n}\omega_{n-1}^{-n}.
\end{equation}
The Petty projection inequality is the geometric foundation for
Zhang's affine Sobolev inequality \cite{Zhang}. Its equivalent forms
and extensions have given rise to fundamental inequalities in
analysis, geometry and information theory, e.g., \cite{LYZ_Lp},
\cite{LYZ_Sobolev}.

The affine invariance in inequalities \eqref{eqn:BS} and
\eqref{eqn:Petty} follows from volumetric considerations.  However, as
we will review below, the underlying principle goes much deeper and
extends to the family of affine quermassintegrals, of which
$\abs{L^{\circ}}$ and $\abs{(\Pi L)^{\circ}}$ are just two special
cases, up to normalization. Formally, the affine quermassintegrals are
defined for compact sets $L\subseteq \mathbb{R}^n$ and $1\leq k\leq n$
by
\begin{equation}
  \label{aff-quer-1-def}
  \Phi_{[k]}(L)=
  \left(\int_{G_{n,k}}\abs{P_EL}^{-n}d\nu_{n,k}(E)\right)^{-\frac{1}{kn}},
\end{equation}
where $G_{n,k}$ is the Grassmannian manifold of $k$-dimensional linear
subspaces equipped with the Haar probability measure
$\nu_{n,k}$. Writing $\abs{L^{\circ}}$ and $\abs{(\Pi L)^{\circ}}$ in
polar coordinates shows a direct connection to $k=1$ and $k=n-1$ in
\eqref{aff-quer-1-def}, respectively.  As the name suggests, they are
affine-invariant, i.e., $\Phi_{[k]}(TL)=\Phi_{[k]}(L)$ for each volume
preserving affine transformation $T$, as proved by Grinberg \cite{Gr},
extending earlier work on ellipsoids by Furstenberg-Tzkoni \cite{FT}
and Lutwak \cite{Lu1}.

The quantities $\Phi_{[k]}(L)$ are affine versions of
quermassintegrals or intrinsic volumes, which play a central role in
Brunn-Minkowski theory \cite{S}.  In particular, the intrinsic volumes
$V_1(L),\ldots, V_n(L)$ of a convex body $L$ admit similar
representations through Kubota's integral recursion as
\begin{equation}
  V_k(L) = c_{n,k}\int_{G_{n,k}}\abs{P_EL}d\nu_{n,k}(E),
\end{equation}
where $c_{n,k}$ is a constant that depends only on $n$ and $k$.  They
enjoy many fundamental inequalities, such as
\begin{equation}
  \label{eqn:Alexandrov}
 V_k(L)\geq V_k(r_L B_2^n), 
\end{equation}
for $k=1,\ldots,n-1$, where $r_L$ is the radius of a Euclidean ball
having the same volume as $L$.  Taking $k=1$ in \eqref{eqn:Alexandrov}
corresponds to Urysohn's inequality, while $k=n-1$ is the standard
isoperimetric inequality. From Jensen's inequality one sees that
\eqref{eqn:BS} and \eqref{eqn:Petty} provide stronger affine-invariant
analogues of \eqref{eqn:Alexandrov} for $k=1$ and $k=n$, respectively.
For the intermediary values $1<k<n$, the inequalities in
\eqref{eqn:Alexandrov} are well-known consequences of
Alexandrov-Fenchel inequality, e.g., \cite{S}. On the other hand, it
is still an open problem, posed by Lutwak \cite{Lu1}, \cite[Problem
  9.3]{GarB}, to determine minimizers for their affine versions,
namely, to prove that for $1<k<n-1$,
\begin{equation}
  \label{eqn:Lutwak_conj}
  \Phi_{[k]}(L)\geq \Phi_{[k]}(r_LB_2^n).
\end{equation}

In the last 40 years, a compelling dual theory, initiated by Lutwak in
\cite{Lut_dual}, has flourished (see, e.g., \cite{S}, \cite{GarB}). Rather than
convex bodies and projections onto lower-dimensional subspaces, this
involves star-shaped sets and intersections with subspaces. As above,
a key isoperimetric inequality lies at its foundation. The
intersection body of a star-shaped body $L$ is the star-shaped body
$IL$ with radial function $\rho_{IL} (\theta) := | L \cap
\theta^{\perp} |$.  The Busemann intersection inequality \cite{Bu},
proved originally for convex bodies $L$, states that
\begin{equation}
  \label{eqn:Busemann}
  \abs{IL}\abs{L}^{-(n-1)}\leq \omega_{n-1}^n\omega_n^{-(n-2)}.
\end{equation}
The volume of the intersection body lies at one end-point of a
sequence of $SL_{n}$-invariant quantities that are called the dual
affine quermassintegrals. These are $SL_{n}$-invariant analogs of the
dual quermassintegrals introduced by Lutwak \cite{Lu00}. Formally, for
a compact set $L\subseteq \mathbb{R}^n$ and $1\leq k\leq n$, the dual
affine quermassintegrals of $L$ are defined by
\begin{equation}
 \label{dual-aff-quer-1-def}
\Psi_{[k]}(L) = \left(\int_{G_{n,k}}\abs{L\cap
  E}^nd\nu_{n,k}(E)\right)^{\frac{1}{kn}}.
\end{equation}
As above, Grinberg \cite{Gr}, drawing on \cite{FT}, showed that these
enjoy invariance under volume-preserving linear transformations, i.e.
$\Psi_{[k]}(TL)= \Psi_{[k]}(L)$ for $T\in SL_{n}$.  They also satisfy
the following extension of \eqref{eqn:Busemann}, proved by
Busemann-Straus \cite{Bu} and Grinberg \cite{Gr}:
\begin{equation}
  \label{Gr-ineq}
  \Psi_{[k]}(L)\leq \Psi_{[k]}(r_L B_2^n).
\end{equation}

While the dual theory has been developed for star-shaped bodies, the
investigation of these quantities goes deeper and can be extended to
bounded Borel sets and non-negative measurable functions \cite{Gar},
\cite{DPP}. For recent developments on dual Brunn-Minkowski
theory, see \cite{S}, \cite{GarB}, \cite{HLYZ} and the references
therein.

The theory that has developed around affine and dual affine
quermassintegrals has implications outside of convex geometry.  As a
sample, we mention variants of \eqref{Gr-ineq} for functions in
\cite{DPP} lead to sharp asymptotics for small-ball probabilities for
marginal densities when independence may be lacking; small-ball
probabilities for the volume of random polytopes \cite{PP2}; bounds on
marginal densities of $\log$-concave measures connected to the Slicing
Problem \cite{PaoVal}. In these applications, the main focus was on
volumetric estimates and implications for high-dimenional probability
measures. Recently, there is increasing interest in other
probabilistic aspects of Grassmannians and flag manifolds such as
topological properties of random sets in real algebraic geometry; see
\cite{BurLer} and the references therein.

\subsection*{Towards flag manifolds}

Given the usefulness of affine and dual affine quermassintegrals, it
is worth re-visiting the role ellipsoids have played in their
development. The work of Furstenberg-Tzkoni \cite{FT} that established
the $SL_{n}$-invariance of \eqref{dual-aff-quer-1-def} for ellipsoids
went well beyond this special case. One aspect of \cite{FT} that has
received less attention is kindred integral geometric formulas for
ellipsoids on flag manifolds. They established deeper connections to
representation of spherical functions on symmetric spaces. Unlike
affine and dual affine quermassintegrals, the corresponding notions
for convex bodies, compact sets or functions have not been
investigated in the setting of flag manifolds. Our main goal is to
initiate such a study in this paper.

Flag manifolds are natural generalizations of Grassmannians in
geometry.  In convex geometry, mixed volumes admit representations in
terms of certain flag measures, e.g., \cite{HRW}.  Our work goes in a
different direction and the focus here is on flag versions of
quantities like those in \eqref{aff-quer-1-def} and
\eqref{dual-aff-quer-1-def} and corresponding extremal inequalities.
We establish fundamental properties such as affine invariance and
affine inequalities. We also treat companion approximate reverse
isoperimetric inequalities, which play an important role in
high-dimensional convex geometry and probability.

\subsection{Main results}

We start by recalling the setting from work of Furstenberg and Tzkoni
\cite{FT}. Let $ 1\leq r \leq n-1$ and let $ {\bf r}:= ( i_{1}, i_{2},
\cdots, i_{r})$ be a strictly increasing sequence of integers, $1\leq
i_{1} < i_{2} < \cdots < i_{r} \leq n-1$. Let $ \xi_{{\bf r}} := (
F_{1}, \cdots , F_{r})$ be a (partial) flag of subspaces; i.e. $ F_{1}
\subset F_{2} \subset \cdots \subset F_{r}$ with each $F_{j} $ an
$i_{j}$-dimensional subspace. We denote by $ F_{{\bf r}}^{n} $ the
flag manifold (with indices ${\bf r}$) as the set of all partial flags
$\xi_{\bf r}$. $F_{{\bf r}}^{n} $ is equipped with the unique Haar
probability measure that is invariant under the action of $SO_{n}$ and
all integrations on this set in this note are meant with respect to
that measure.

In the special case when $r=1$ and $i_{1}=k$, the partial flag
manifold $F^{n}_{{\bf r}}$ is just the Grassmann manifold
$G_{n,k}$. Hence the (partial) flag-manifolds can be considered as
generalizations of Grassmannians.  When $r=n-1$, so that $ {\bf r}:=
(1, 2, \ldots, n-1)$, we write $F^{n}:= F_{{\bf r}}^{n}$ for the
complete flag manifold. 
We follow the convention that $i_{0} = 0$ and $i_{r+1}= n$, hence
\begin{equation}
\label{basic-id-0}
\sum_{j=1}^{r} i_{j} ( i_{j+1} - i_{j-1}) = i_{r} n.
\end{equation}

Let $L$ be a compact set in $\mathbb R^{n}$ and let $ 1\leq r \leq
n-1$ and ${\bf r}$ be a set of indices as above. We define the {\it
  ${\bf r}$-flag quermassintegral} of $L$ by
\begin{equation}
  \label{r-flar-quer}
  \Phi_{{\bf r}} (L) := \left( \int_{F_{{\bf r}}^{n}} \prod_{j=1}^{r} |
  P_{F_{j}} L |^{ i_{j-1}- i_{j+1}} d\xi_{{\bf r}}
  \right)^{-\frac{1}{i_{r} n}}.
\end{equation}
Similarly, we define the {\it dual ${\bf r}$-flag quermassintegral} of
$L$ by
\begin{equation}
  \label{dual-r-flar-quer}
  \Psi_{{\bf r}} (L) := \left( \int_{F_{{\bf r}}^{n}} \prod_{j=1}^{r} | L\cap F_{j} |^{ i_{j+1}- i_{j-1}} d\xi_{{\bf r}} \right)^{\frac{1}{i_{r} n}}.
\end{equation}
In \cite{FT}, it was shown that when $L=\cal{E}$ is an ellipsoid,
$\Psi_{{\bf r}}(\cal{E})$ is invariant under $SL_{n}$. When $r=1$, the
${\bf r}$-flag quermassintegrals are exactly the affine
quermassintegrals; similarly for the dual case.  Thus the latter
quantities can be considered as extensions of the (dual) affine
quermassintegrals to flag manifolds. For complete flag manifolds, we
similarly define
\begin{equation}
  \label{fullFlag} 
  \Psi_{{\bf F^{n}}}(L) := \left( \int_{F^{n}} \prod_{i=1}^{n-1} |L
  \cap F_{i} |^{2} d \xi\right)^{\frac{1}{n(n-1)}} {\rm and } \quad
  \Phi_{{\bf F^{n}}}(L) := \left( \int_{F^{n}} \prod_{i=1}^{n-1} |
  P_{F_{i}} L |^{2} d \xi\right)^{\frac{1}{n(n-1)}} .
\end{equation}
Clearly, by (\ref{basic-id-0})
\begin{equation}
\label{homogen}
\Psi_{{\bf r}} ( \lambda L) = \lambda \Psi_{{\bf r}} ( L), \quad \quad
\ \Phi_{{\bf r} } ( \lambda L) = \lambda \Phi_{{\bf r} } ( L)\quad
(\lambda >0).
\end{equation}

Our first result extends the invariance results of Grinberg \cite{Gr}
that give invariance of \eqref{aff-quer-1-def} and
\eqref{dual-aff-quer-1-def} under volume-preserving affine and linear
transformations, respectively.

\begin{theorem}\label{thm-aff}
  Let $L$ be a compact set in $\mathbb R^{n}$, $ 1\leq r \leq n-1$ and
  $ {\bf r} := ( i_{1}, \cdots , i_{r})$ be an increasing sequence of
  integers between $1$ and $n-1$. Let $A$ be an affine map that
  preserves volume and $T\in SL_{n}$. Then
  \begin{equation}
    \label{thm-aff-1}
    \Phi_{{\bf r}} ( A L ) = \Phi_{{\bf r}} ( L ) \quad {\rm and }
    \quad \Psi_{{\bf r}} ( T L ) = \Psi_{{\bf r}} ( L ) .
  \end{equation}
\end{theorem}

With such invariance properties, it is natural to seek extremizers of
$\Phi_{{\bf r}}(L)$ and $\Psi_{{\bf r}}(L)$, especially over convex
bodies $L \subseteq \mathbb{R}^n$. However, even for the Grassmannian
very few such results are known; cf. Lutwak's conjectured inequality
\eqref{eqn:Lutwak_conj}.  We note, however, that inequality
\eqref{eqn:Lutwak_conj} does hold at the expense of a universal
constant, as proved by the second and third-named authors \cite{PP2}.
It is easy to construct compact sets $L \subseteq \mathbb{R}^n$ of a
given volume such that $\Phi_{[k]}(L)$ is arbitrarily large. This,
however, cannot happen when $L$ is convex: in \cite{DP} it was shown
that up to a logarithmic factor in the dimension $n$, $\Phi_{[k]}(L)$
does not exceed $\Phi_{[k]}(r_L B_2^n)$.

We extend the aforementioned results to the setting of ${\bf r}$-flag
quermassintegrals. In this note $c, c^{\prime}, c_{0},\cdots$
etc. will denote universal constants (not necessarily the same at each
occurrence).

\begin{theorem}\label{thm-aff-ineq}
  Let $L$ be a compact set in $\mathbb R^{n}$, $ 1\leq r \leq n-1$ and
  ${\bf r}:= (i_{1}, \cdots , i_{r})$ an increasing sequence of integers
  between $1$ and $n-1$. Then
  \begin{equation}
    \label{ineq-0-1}
    \Psi_{{\bf r}} (L) \leq \Psi_{{\bf r}} (r_L B_2^n).
  \end{equation}
  If $L$ is a symmetric convex body, then
  \begin{equation}
    \label{ineq-0-2}
    \Psi_{{\bf r}} (L) \geq \frac{ c}{\min\left\{ \sqrt{\frac{n}{i_{r}}} ,
      \log{n}\right\} }\Psi_{{\bf r}} (r_L B_2^n).
  \end{equation} 
  If $L$ is a convex body, then
  \begin{equation}
    \label{ineq-0-3}
    \frac{1}{c} \Phi_{{\bf r}} (r_L B_2^n) \leq \Phi_{{\bf r}} (L) \leq c
    \min\left\{ \sqrt{\frac{n}{i_{r}}} , \log{n}\right\} \Phi_{{\bf r}}
    (r_L B_2^n) .
  \end{equation} 
\end{theorem}

Further drawing on \cite{FT}, we also consider variants of ${\bf
  r}$-flag (dual) affine quermassintegrals involving permutations $\omega$ of
$\{1,\ldots,n\}$. We define the {\it $\omega$-flag quermassintegral}
and {\it $\omega$-flag dual quermassintegral} as follows: for every
compact set $L$ in $\mathbb R^{n}$,
\begin{equation}
\Phi_{\omega} (L) := \begin{cases} \left( \int_{{\bf F^{n}}}
  \prod_{j=1}^{n-1} | P_{F_{j}}L|^{- \omega(j) + \omega(j+1)-1} d \xi
  \right)^{-\frac{1}{n( n- \omega(n))}} &\ {\rm if } \ \omega(n) \neq
  n,\\
 \int_{{\bf F^{n}}} \prod_{j=1}^{n-1} | P_{F_{j}}L|^{-
  \omega(j) + \omega(j+1)-1} d \xi &\ {\rm if } \ \omega(n) = n.
 \end{cases}
\end{equation}
and 
\begin{equation}
\Psi_{\omega} (L) := \begin{cases}\left( \int_{{\bf F^{n}}}
  \prod_{j=1}^{n-1} | L \cap F_{j}|^{ \omega(j) - \omega(j+1)+1} d \xi
  \right)^{\frac{1}{ n(n- \omega(n))}} &\ {\rm if} \ \omega(n) \neq
  n,\\ \int_{{\bf F^{n}}} \prod_{j=1}^{n-1} | L \cap F_{j}|^{ \omega(j) -
    \omega(j+1)+1} d \xi &\ {\rm if } \ \omega(n) = n.
  \end{cases}
\end{equation}
Furstenberg and Tzkoni showed $SL_{n}$-invariance of $\Psi_{\omega}$
for ellipsoids. We investigate the extent to which this invariance
carries over to compact sets.  Moreover, in the case of convex bodies
we show that such quantities cannot be too degenerate in the sense
that they admit uniform upper and lower bounds, independent of the
body. We apply V. Milman's $M$-ellipsoids \cite{Mil88}, together with
the aforementioned $SL_n$-invariance of Furstenberg-Tzkoni to
establish these bounds (see Corollary \ref{M-FT}).

In \S \ref{section:functions}, we introduce functional analogues of
the ${\bf r}$-flag dual affine quermassintegrals. We show that more
general quantities share the $SL_{n}$-invarinace properties and we
prove sharp isoperimetric inequalities. In this section, we invoke
techniques and results from our previous work \cite{DPP}. Lastly, in
\S \ref{section:functions} we also introduce a functional form of
${\bf r}$-flag affine quermassintegrals. There is much recent interest
in extending fundamental geometric inequalities from convex bodies to
certain classes of functions, e.g., \cite{KlM}, \cite{BCF},
\cite{MR}. The latter authors have studied variants of inequalities
for intrinsic volumes, or even mixed volumes, and other general
quantities; for example, they establish functional analogues of
\eqref{eqn:Alexandrov}.  Of course, for functions one cannot hope for
a sharp analogue of \eqref{eqn:Lutwak_conj}, as this is open even for
affine quermassintegrals of convex bodies. On the other hand, we
establish a general functional result at the expense of a universal
constant. Invariance properties and bounds for these quantities are
treated in \S \ref{sub:func}.

\section{Affine invariance}

In this section we will present the proof of Theorem \ref{thm-aff}. The following proposition relates integration on a flag manifold to integration on nested Grassmannians, (see \cite{SW}
Theorem 7.1.1 on p. 267 for such a result for flags of elements consisting of two subspaces). Since we will use this fact many times throughout this paper, we include the proof. For a subspace $F \subset \mathbb{R}^n$, we denote by $G_{F,i}$ the Grassmannian of all $i$-dimensional subspaces contained in $F$.

\begin{proposition}
Let $ 1\leq r \leq n-1$ and	$ {\bf r} := ( i_{1}, \cdots , i_{r})$ be an increasing sequence of integers between $1$ and $n-1$. For $G \in L^1(F^n_{\bf r})$,
  \begin{equation}\label{basic-flags-1}
    \int\limits_{F^n_{\bf r}} G (\xi_{\bf r}) d\xi_{\bf r} = \int\limits_{G_{n,i_{r}} }
    \int\limits_{G_{F_{r}, i_{r-1}}} \cdots \int\limits_{G_{ F_{2}},
      i_{1}} G ( F_{1}, \cdots , F_{r}) \, d F_{1}
    \cdots d F_{r-1} dF_{r}.
  \end{equation}
\end{proposition}

For simplicity, we have suppressed the notation to write $dF_1$ rather
than $d\mu_{G_{F_2,i_1}}(F_1)$; similarly for all other indices. This
convention will be used throughout.

\begin{proof}
Fix $i_j$. Denote by $SO(F_{j})$ the subgroup of $SO_n$ acting
transitively on $G_{F_{j},{i_{j-1}}}$. For example, if
$F_o=\mathop{\rm span}\{e_1, \dots, e_{i_j}\}$ and $E_o=\mathop{\rm span}\{e_1,
\dots, e_{i_{j-1}}\}$, then elements of $SO(F_o)$ are given by
\[ \left( \begin{array}{cc} SO_{i_j} & 0  \\ 0 & I_{n-i_j} \end{array} \right) .\]
And the stabilizer of $E_o$ in $SO(F_o)$ is
\[ \left( \begin{array}{ccc} SO_{i_{j-1}} & 0 & 0 \\ 0 & SO_{i_j - i_{j-1}} & 0 \\ 0 & 0 & I_{n-i_j} \end{array} \right) .\]
The measure $\mu_{G_{F_j},i_{j-1}}$ is invariant under $SO(F_{j})$. Further, for $g\in SO_n$ and a Borel subset $A \subset G_{F_j,i_{j-1}}$ we have $ \mu_{G_{g F_j,i_{j-1}}} (gA) = \mu_{G_{F_j,i_{j-1}}} (A)$.
	
We will show that both integrals are invariant under the action of $SO_n$. Fix $g\in SO_n$. We start with the integral on the right-hand side of (\ref{basic-flags-1}):
	
\begin{align*}
  \int_{G_{n,i_r} } & \int_{G_{ F_r, i_{r-1} }} \cdots \int_{G_{ F_2,i_1 }} G (g^{-1} \cdot(F_{1}, \cdots, F_{r})) \, d F_1\cdots dF_{r-1} dF_r \\ 
  &= \int_{G_{n,i_r} } \int_{G_{ g F_r, i_{r-1} }} \cdots \int_{G_{ g F_2, i_1 }} G (F_{1}, \cdots, F_{r}) \, d (g F_1)\cdots d (g F_{r-1}) d (g F_r) \\ 
  &= \int_{G_{n,i_r} } \int_{G_{g F_r, i_{r-1} }} \cdots \int_{G_{ F_2, i_1 }} G (F_{1}, \cdots,
  F_{r}) \, dE_1\cdots d (g F_{r-1}) d (g F_r) \\ 
  &= \cdots \\ 
  &=\int_{G_{n,i_r} } \int_{G_{ F_r, i_{r-1} }} \cdots \int_{G_{ F_2, i_1 }} G (F_{1}, \cdots, F_{r}) \, d F_1\cdots d F_{r-1} d F_r,
\end{align*}
where we have sent $(F_{1}, \cdots , F_{r}) \to g \cdot(F_{1}, \cdots, F_{r})$ and then used the invariance property 
$$\mu_{G_{g F_j, i_{j-1}}} (gA) = \mu_{G_{F_j, i_{j-1}}} (A)$$ 
for all $r-1$ inner integrals, for the outer integral we use the $SO_n$-invariance of the measure $\mu_{G_{n,i_r}}$. Note that at each step $(F_{1}, \cdots, F_{r})$ remains an element of $F^n_{\bf r}$, this is to say that the inclusion relation is preserved. The invariance of the integral on the
left-hand side of (\ref{basic-flags-1}) is a consequence of the $SO_n$-invariance of the measure $\mu_{F^n_{\bf r}}$. The proposition now follows by the uniqueness of the $SO_n$-invariant probability measure on $F^n_{\bf r}$, see for example \S 13.3 in \cite{SW}.
\end{proof}

The following fact allows one to view an integral of a function on a
partial flag as an integral over the full flag manifold. In this case,
to avoid confusion, the subspaces of flag manifolds are indexed by
their dimension.

\begin{proposition}  
Let $ 1\leq r \leq n-1$ and	$ {\bf r} := ( i_{1}, \cdots , i_{r})$ be an increasing sequence of integers between $1$ and $n-1$. For a function $ G $ on the partial flag $F_{{\bf r}}^{n}$, denote by $\wtd{G}$ its trivial extension to the full flag manifold $F^{n}$, i.e., $\wtd{G}(F_1, \ldots, F_{n-1}) := G(F_{i_1}, \ldots, F_{i_r})$. Then
\begin{equation}\label{inv-proj}
	\int_{F^{n}} \wtd{G}( \eta) d \eta = \int_{F_{{\bf r}}^{n}} G(\xi_{\bf r}) d \xi_{\bf r}.
\end{equation}
\end{proposition}

\begin{proof} 
We ``integrate out'' the Grasmannians that do not contain subspaces
that $G$ depends on by repeatedly using the identity
$$ \int_{G_{ F_{j+1}, j}} \int_{G_{ F_j, {j-1}}} f(F_{j-1}) d F_{j-1}
d F_{j} = \int_{G_{ F_{j+1}, j-1}} f(F_{j-1}) dF_{j-1}. $$ On the
right-hand side, we integrate over the set of all $(j-1)$-dimensional
subspaces in the ambient $(j+1)$-dimensional space. On the left-hand
side we integrate over the same set of planes stepwise, we step from
one $j$-dimensional subspace in the ambient $(j+1)$-dimensional space to the next and in each such subspace we consider all $(j-1)$-dimensional subspaces. The above identity holds
since we are using probability measures on each nested Grassmannian.
Applying the latter iteratively, we get
\begin{align*}
  & \int_{F^{n}} \wtd{G}( \eta) d \eta \\ 
  &= \int_{G_{n,n-1} } \int_{G_{ F_{n-1}, n-2 }} \cdots \int_{G_{ F_2, 1 }} \wtd{G} (F_{1},
  \cdots , F_{n-1}) d F_1\cdots d F_{n-2} d F_{n-1} \\ 
  &= \int_{G_{n,n-1} } \int_{G_{ F_{n-1}, n-2 }} \cdots \int_{G_{ F_2, 1}} 
  G (F_{i_1}, \ldots, F_{i_r}) dF_1\cdots d F_{n-2} d F_{n-1} \\ 
  & = \int_{G_{n,n-1} } \cdots \int_{G_{ F_{i_1+1}, i_1 }} G (F_{i_1}, \ldots, F_{i_r}) \left( \int_{G_{F_{i_1},i_1-1}} \cdots \int_{G_{F_2,1}} dF_1 \ldots
  dF_{i_1-1} \right) d F_{i_1} \cdots dF_{n-1} \\ 
  &=  \int_{G_{n,n-1} } \cdots \int_{G_{ F_{i_1+1}, i_1 }} G (F_{i_1}, \ldots, F_{i_r}) d F_{i_1} \cdots dF_{n-1} \\ 
  &= \int_{G_{n,n-1} } \cdots \left( \int_{G_{ F_{i_2}, i_2-1 }} \cdots \int_{G_{ F_{i_1+1}, i_1}} G(F_{i_1}, \ldots, F_{i_r}) d F_{i_1} \cdots dF_{i_2-1} \right)
  \cdots d F_{n-1} \\ 
  &= \int_{G_{n,n-1} } \cdots \left( \int_{G_{F_{i_2}, i_1 }} G (F_{i_1}, \ldots, F_{i_r}) d F_{i_1} \right) \cdots dF_{n-1} \\ 
  &= \cdots \\ 
  &= \int_{G_{n,i_r} } \int_{G_{F_{i_r}, i_{r-1} }} \cdots \int_{G_{ F_{i_2}, i_1 }} G (F_{i_1}, \cdots, F_{i_r}) d F_{i_1}\cdots dF_{i_{r-1}}d F_{i_r} \\ 
  &=\int_{F_{{\bf r}}^{n}} G(\xi_{\bf r}) d \xi_{\bf r}.
\end{align*}	
\end{proof}

%

We now turn to the invariance properties of the functionals
$\Phi_{{\bf r}}$ and $\Psi_{{\bf r}}$. Although self-contained proofs
are possible, they require somewhat involved machinery. Since all of the
ingredients are available in the literature \cite{FT, Gr, DPP}, we
have chosen to gather the essentials without proofs. 

For readers less familiar with the relevant work, we will explain the main points
behind the affine invariance of the functionals $\Phi_{{[k]}}(K)$ and
$\Psi_{{[k]}}(K)$ along the way. There are two important changes of
variables: a `global' change of variables on the Grassmannian
$G_{n,k}$ or the flag manifold $F^n_{\bf r}$ and a `local' change of
variables on each element $F\in G_{n,k}$ or $\xi_{\bf r} \in F^n_{\bf r}$.

Let $g\in SL_n$, $F\in G_{n,k}$ and $A\subset F$ be a
full-dimensional Borel set, then $|g A| = |\det(g|_F)| |A|$. This
determinant of the transformation $g$ restricted to the subspace $F$,
$\det(g|_F)$, is the Jacobian in the following change of variables:
\begin{equation}\label{eq_lcv}
\int_{gF} f(g^{-1}t) dt = \int_F f(t) |\det(g|_F)| dt.
\end{equation}
Denote it as in \cite{FT} by $\sigma_k(g,F):= |\det(g|_F)| = \frac{|g
  A|}{|A|}$.

For the relevant manifolds $M$ considered in this paper, denote by
$\sigma_M(g,F)$ the Jacobian determinant in the following change of
variables:
\begin{equation}\label{eq_gcv}
\int_M f(F) dF = \int_M f(gF) \sigma_M(g,F) dF . 
\end{equation}
Furstenberg and Tzkoni proved in \cite{FT} that
 \begin{equation}
 \label{G-S-03}
  \sigma_{G_{n,k}}( g, F) = \sigma_{k}^{-n} (g,F)
 \end{equation}
 and 
\begin{equation}
\label{F-S-02} 
\sigma_{F_{{\bf r}}^{n} }(g, \xi_{\bf r}) = \sigma_{i_{1}}^{-i_{2}}(g, F_1)
\sigma_{i_{2}}^{i_{1}-i_{3}}(g, F_2) \cdots \sigma_{i_{r}}^{i_{r-1}-n}(g, F_r) ,
\end{equation}
where ${\bf r}:= (i_{1}, \cdots , i_{r})$.  The linear-invariance of
the dual affine quermassintergals $ \Psi_{[k]} $ now follows
immediately. Indeed, for $g\in SL_n$
\begin{align*}
\Psi^{kn}_{[k]}(gL) &=\int_{G_{n,k}} |gL\cap F|^n dF = \int_{G_{n,k}}
|gL\cap gF|^n \sigma_{G_{n,k}}(g,F) dF \\ &= \int_{G_{n,k}}
\left(\sigma_{k}(g,F) |L\cap F|\right)^n \sigma_{k}^{-n} (g,F) dF =
\Psi^{kn}_{[k]}(L),
\end{align*}
where we have used (\ref{eq_gcv}), (\ref{eq_lcv}) with $f=1_L$ and
(\ref{G-S-03}). Now we turn toward the proof of Theorem
\ref{thm-aff}. We start with the case of dual ${\bf r}$-flag
quermassintegrals.

\begin{proposition}
Let $ 1\leq r \leq n-1$ and	$ {\bf r} := ( i_{1}, \cdots , i_{r})$ be an increasing sequence of	integers between $1$ and $n-1$.
For every compact set $L$ in $\mathbb R^{n}$ and every $g \in SL_{n}$,
\begin{equation}
\label{sln-inv}
\Psi_{{\bf r}} (g L) = \Psi_{{\bf r}} (L)   .
\end{equation}
\end{proposition}

\begin{proof}
Let us start by expressing $\sigma_{F_{{\bf r}}^{n} }(g, \xi_{\bf r})$ in terms of
sections. For this note that
$$ \sigma_{i_{j}}(g, F_{j}) = \frac{ |g( L \cap F_{j})|}{ | L\cap
  F_{j}|}, $$ where as a subset of $F_{j}$ we use the section $ L \cap
F_{j}$. By (\ref{F-S-02}) with $i_0=0$ and $i_{r+1}=n$, we have
$$ \sigma_{F_{{\bf r}}^{n} } ( g, \xi_{\bf r}):= \prod_{j=1}^{r}
\sigma_{i_{j}}^{-i_{j+1}+ i_{j-1}} (g, F_{j})=\prod_{j=1}^{r} 
\frac{ |L \cap F_{j}|^{ i_{j+1}- i_{j-1}} }{ |gL \cap g F_{j}|^{ i_{j+1}-i_{j-1}}} . $$ 
Using the change of variables (\ref{eq_gcv}) with
the above expression for $\sigma_{F_{{\bf r}}^{n} }$, yields
\begin{align*}
\Psi_{{\bf r}}^{ni_{r}} ( gL) 
&= \int_{F_{\bf r}^{n}} \prod_{j=1}^{r} | gL \cap F_{j}|^{ i_{j+1}- i_{j-1}} d\xi_{\bf r} \\ 
&= \int_{F_{\bf r}^{n}} \prod_{j=1}^{r} | gL \cap g F_{j}|^{ i_{j+1}- i_{j-1}}
\sigma_{F_{{\bf r}}^{n}} (g,\xi_{\bf r}) d\xi_{\bf r} \\ 
&= \int_{F_{\bf r}^{n}} \prod_{j=1}^{r} | gL \cap g F_{j}|^{ i_{j+1}- i_{j-1}}\prod_{j=1}^{r} \frac{ | L \cap F_{j}|^{ i_{j+1}- i_{j-1}} }{ |gL \cap g F_{j}|^{i_{j+1}- i_{j-1}}} d\xi_{\bf r} \\ 
&=\int_{F_{\bf r}^{n}} \prod_{j=1}^{r} | L \cap F_{j}|^{ i_{j+1}- i_{j-1}} d\xi_{\bf r} \\ 
&=\Psi_{{\bf r}}^{ni_{r}} (L) .
\end{align*}
\noindent This proves \eqref{sln-inv}. 
\end{proof}

To recall the linear-invariance of the operator $\Phi_{[k]}$, we again
follow Grinberg \cite{Gr}. Observe that for $F\in G_{n,k}$ and $g\in
SL_n$ upper-triangular with respect to the decomposition
$\mathbb{R}^n=F+F^{\perp}$, we have
$$ |P_F ( g^t L )| = |g P_F L| = |\det(g|_F)| |P_F L| = \sigma_k(g,F)
|P_F L|. $$ 
While for $l\in SO_n$, we have $P_F(l^t L)=P_{l F}(L)$. Since
any $g\in SL_n$ can be written as a product of a rotation and an
upper-triangular matrix, combining the two observations yields the following.

\begin{lemma}[\cite{Gr}]
\noindent Let $L$ be a compact set in $\mathbb R^{n}$, $F\in G_{n,k}$ and $g \in SL_{n}$. Then
\begin{equation}
\label{multiplier-grinberg}
| P_{F} (g^t L)| = |P_{gF} L| \sigma_{k}(g,F) .
\end{equation}
\end{lemma}

\noindent The linear-invariance of the affine quermassintergals $
\Phi_{[k]} $ can now be seen as follows: let $g\in SL_n$
$$ \Phi^{-kn}_{[k]}(g^t L)=\int_{G_{n,k}} | P_{F} (g^t L)|^{-n} dF = \int_{G_{n,k}} |P_{gF} L|^{-n} \sigma_{k}^{-n} (g,F) dF = \Phi^{-kn}_{[k]} (L), $$
where we have used (\ref{multiplier-grinberg})  and (\ref{eq_gcv}) taking into account (\ref{G-S-03}).

\begin{proposition}
Let $ 1\leq r \leq n-1$ and $ {\bf r} := ( i_{1}, \cdots , i_{r})$ be
an increasing sequence of integers between $1$ and $n-1$.  Let $A$ be
an affine volume preserving map in $\mathbb R^{n}$. Then for every
compact set $L$ in $\mathbb R^{n}$,
\begin{equation}
\label{aff-inv}
\Phi_{{\bf r}} (AL) = \Phi_{{\bf r}} (L) .
\end{equation}
\end{proposition}

\begin{proof}
\noindent We will first prove the theorem in the case $A:= g\in SL_{n}$. Using (\ref{multiplier-grinberg}) for the projection onto each $F_j$, (\ref{F-S-02}) and making the change of variables (\ref{eq_gcv}), we get
\begin{align*}
\Phi^{-n i_r}_{{\bf r}} ( g^t L) 
&=  \int_{F_{\bf r}^{n}} \prod_{j=1}^{r} | P_{F_{j}}(g^t L)|^{- i_{j+1}+ i_{j-1}} d\xi_{\bf r}  \\
&=  \int_{F_{\bf r}^{n}} \prod_{j=1}^{r} \left( |P_{g F_j} L| \sigma_{i_j}(g,F_j)  \right)^{- i_{j+1}+ i_{j-1}} d\xi_{\bf r}  \\
&=  \int_{F_{\bf r}^{n}} \prod_{j=1}^{r} |P_{g F_j} L|^{- i_{j+1}+ i_{j-1}}  \prod_{j=1}^{r} \sigma^{- i_{j+1}+ i_{j-1}}_{i_j}(g,F_j)  d\xi_{\bf r}  \\
&=  \int_{F_{\bf r}^{n}} \prod_{j=1}^{r} |P_{g F_j} L|^{- i_{j+1}+ i_{j-1}}  \sigma_{F_{{\bf r}}^{n} } ( g, \xi_{\bf r})  d\xi_{\bf r} \\
&=  \int_{F_{\bf r}^{n}} \prod_{j=1}^{r} | P_{F_{j}} L|^{- i_{j+1}+i_{j-1}} d\xi_{\bf r} \\
&=  \Phi^{-n i_r}_{{\bf r}} (L) . 
\end{align*}
\noindent The general case follows easily. 
\end{proof}

\noindent  The proof of Theorem \ref{thm-aff} is now complete.

\section{Inequalities}
\label{section:inequalities}

We start by proving an extension of the inequality of Busemann-Straus
and Grinberg \eqref{Gr-ineq} to flag manifolds.

\begin{proposition}
\label{basic-ineq}
Let $ 1 < r \leq n-1$ and	$ {\bf r} := (i_{1}, \cdots, i_{r})$ be an increasing sequence of	integers between $1$ and $n-1$.
Then for every compact set $L$ in $\mathbb{R}^n$,
\begin{equation}
  \label{Psi-ineq}
  \Psi_{{\bf r}} (L)  \leq \Psi_{{\bf r}} (r_{L    } B_{2}^{n}) 
\end{equation}
with equality if and only if $L$ is a centered ellipsoid (up to a set
of measure zero).
\end{proposition}

\begin{proof}
Inequality \eqref{Gr-ineq} implies that for every $n$, every $ E\in G_{n,m}$, any $1\leq \ell \leq m-1$ and every compact set $L \subseteq \mathbb R^{n}$
\begin{equation}
	\label{eq-00-1}
 	\int_{G_{E,\ell} } | (L \cap E) \cap F |^{ m} d \mu_{G_{E, \ell}} (F)
 	\leq c_{m,\ell} | L \cap E|^{\ell} ,
\end{equation}
with equality iff $L \cap E $ is an ellipsoid (up to a measure $0$ set
- see \cite{Gar}). Using \eqref{basic-flags-1} and \eqref{eq-00-1} we
have that
\begin{eqnarray*}
\Psi_{\bf r}^{ i_{r} n} (L) & = &\int_{F_{\bf r}^{n}} \prod_{j=1}^{r}
|L\cap F_{{j}} |^{ i_{j+1}- i_{j-1}} d \xi_{\bf r} \\ & = & \int_{
  G_{n, i_{r}} } \int_{G_{F_{r}, i_{r-1}}} \cdots \int_{ G_{F_{2},
    i_{1}}} \prod_{j=1}^{r} |L\cap F_{{j}} |^{ i_{j+1}- i_{j-1}}
dF_{1} \cdots dF_{r-1} dF_{r} \\ & = & \int_{ G_{n, i_{r}} }
\int_{G_{F_{r}, i_{r-1}}} \cdots\int_{G_{F_{3}, i_{2}}}
\prod_{j=2}^{r} |L\cap F_{{j}} |^{ i_{j+1}- i_{j-1}} \times \\ & &
\;\; \times \left( \int_{ G_{F_{2}, i_{1}}} |L\cap F_{1} |^{ i_{2}}
dF_{1} \right) dF_{2} \cdots dF_{r-1} dF_{r}\\ & = & \int_{ G_{n,
    i_{r}} } \int_{G_{F_{r}, i_{r-1}}} \cdots\int_{G_{F_{3}, i_{2}}}
\prod_{j=2}^{r} |L\cap F_{{j}} |^{i_{j+1}- i_{j-1}} \times \\ & & \;\;
\times \left( \int_{ G_{F_{2}, i_{1}}} |(L\cap F_{2}) \cap F_{1} |^{
  i_{2}} dF_{1} \right) dF_{2} \cdots dF_{r-1} dF_{r} \\ & \leq &
c_{i_{2}, i_{1}} \int_{ G_{n, i_{r}} } \int_{G_{F_{r}, i_{r-1}}}
\cdots\int_{G_{F_{3}, i_{2}}} \prod_{j=2}^{r} |L\cap F_{{j}} |^{
  i_{j+1}- i_{j-1}} \times \\ & & \;\; \times |L \cap F_{2} |^{ i_{1}}
dF_{2} \cdots dF_{r-1} dF_{r} \\ &\leq& \cdots \\ &\leq & |L|^{i_{r}}
\prod_{j=1}^{r} c_{i_{j+1}, i_{j}}.
\end{eqnarray*}

The last inequality is an equality only when $L$ is a centered
ellipsoid, up to set of measure zero; see \cite{Gar}. Since for the
Euclidean ball all inequalities in the previous chain are actually
equalities, we can compute the constants and by the linear-invariance property established by Furstenberg-Tzkoni we conclude the proof.
\end{proof}

Our next result is a type of Blaschke-Santal\'{o} and reverse
Blaschke-Santal\'{o} inequality for ${\bf r}$-flag
quermassintegrals. These inequalities concern the volume of the polar
body. For a compact set $L$ we define the {\it polar body} $L^{\circ}$
(with respect to the origin) as the convex body
$$ L^{\circ} := \{ x\in \mathbb R^{n} : \langle x, y\rangle\leq 1 ,
\ \forall y \in L \}. $$
It is straightforward to check the following
inclusion: for every compact set $L$ in $\mathbb R^{n}$ and $F\in
G_{n,k}$,
\begin{equation}\label{proj-section-inc}
  P_{F} L^{\circ}  \subseteq \left(L\cap F\right)^{\circ}  .
\end{equation} 
If, in addition, $L$ is convex and $0$ is in the interior of $L$, 
\begin{equation}\label{proj-section}
  P_{F} L^{\circ}  = \left( L\cap F\right)^{\circ}  .
\end{equation} 
Recall that the Blaschke-Santal\'{o} inequality (for symmetric convex
bodies), e.g., \cite{GarB}, \cite{S}, states that for every symmetric
convex body $L$ in $\mathbb R^{n}$,
\begin{equation}
  \label{san}
  |L| | L^{\circ} | \leq |B_{2}^{n}|^{2} .
\end{equation}
Moreover \eqref{san} holds when $L$ is convex and $L^{\circ}$ is
centered \cite{S}. An approximate reverse form of this inequality is
known as the Bourgain-Milman theorem \cite{BM}: for every compact,
convex set $L$ with $ 0 \in {\rm int}(L)$,
\begin{equation}
\label{r-san}
|L| |L^{\circ} | \geq  c^{n}|B_{2}^{n}|^{2}.
\end{equation}
Other proofs of this inequality include \cite{Mil88}, \cite{Kup08},
\cite{Naz12}, \cite{GPV14}.

The next proposition is the aforementioned Blaschke-Santal\'{o} and
its (approximate) reversal in the setting of ${\bf r}$-flag manifolds:

\begin{proposition}
  \label{prop:BS_flag}
Let $ 1 \leq r \leq n-1$ and $ {\bf r} := ( i_{1}, \cdots, i_{r})$ be
an increasing sequence of integers between $1$ and $n-1$. Then for a
symmetric compact set $L$ in $\mathbb R^{n}$
  \begin{equation}\label{S for Psi}
    \Phi_{{\bf r}} (L^{\circ}) \Psi_{{\bf r}} (L) \leq \Phi_{{\bf r}}
    ( B_{2}^{n} ) \Psi_{{\bf r}} (B_{2}^{n}).
  \end{equation}
  Moreover, if $L$ is a convex body in $\mathbb R^{n}$ with $ 0 \in {\rm
    int}(L)$, we have that
  \begin{equation}
    \label{RS for Psi}
    \Phi_{{\bf r}} (L^{\circ}) \Psi_{{\bf r}} (L) \geq c \, \Phi_{{\bf r}} ( B_{2}^{n} ) \Psi_{{\bf r}} (B_{2}^{n}) ,
  \end{equation}
  where $c>0$ is an absolute constant - exactly the constant of the reverse Santal\'o inequality \eqref{r-san}. 
\end{proposition}

\begin{proof}
First note that
  $$ \Phi_{{\bf r}}^{i_{r} n} ( B_{2}^{n} ) \Psi_{{\bf r}}^{i_{r} n}
  (B_{2}^{n}) = \left( \prod_{j=1}^{r} | B_{2}^{n} \cap
  F_j|^{i_{j+1}- i_{j-1}} \right)^{2}. $$
Using the Blaschke-Santal\'{o} inequality \eqref{san} and
  \eqref{proj-section-inc}, we have
  \begin{eqnarray*}
    \Psi_{{\bf r}}^{i_{r}n} (L) 
    &=&\int_{F_{{\bf r}}^n }\prod_{j=1}^{r}|L \cap F_j|^{ i_{j+1}- i_{j-1}} d \xi_{\bf r} \\ 
    &\leq & \left(\prod_{j=1}^{r} | B_{2}^{n} \cap F_j|^{i_{j+1}- i_{j-1}}\right)^{2}
    \int_{F_{{\bf r}}^n }\prod_{j=1}^{r} \frac{1}{ |(L \cap F_j)^{\circ}|^{ i_{j+1}- i_{j-1}} } d \xi_{\bf r} \\ 
  	& \leq & \left( \prod_{j=1}^{r} | B_{2}^{n} \cap F_j|^{i_{j+1}-i_{j-1}}\right)^{2}
  	\int_{F_{{\bf r}}^n}\prod_{j=1}^{r} \frac{1}{|P_{F_j}L^{\circ}|^{ i_{j+1}- i_{j-1}} } d \xi_{\bf r} \\ 
  	& = & \Phi_{{\bf r}}^{i_{r} n} ( B_{2}^{n} ) \Psi_{{\bf r}}^{i_{r} n}
    (B_{2}^{n}) \Phi_{{\bf r}}^{-i_{r}n} (L^{\circ}) .
\end{eqnarray*}
On the other hand, using the reverse Blaschke-Santal\'{o} inequality \eqref{r-san}, \eqref{proj-section} and \eqref{basic-id-0} we get
\begin{eqnarray*}
 \Psi_{{\bf r}}^{i_{r}n}(L)
 & = &\int_{F_{\bf r}^{n}}\prod_{j=1}^{r} |L \cap F_j|^{ i_{j+1}- i_{j-1}} d \xi_{\bf r}\\ 
 &\geq & \left( \prod_{j=1}^{r} | B_{2}^{n} \cap F_j|^{i_{j+1}- i_{j-1}} \right)^{2} c^{ \sum_{j=1}^{r} i_{j} (i_{j+1}-i_{j-1}) } \int_{F_{{\bf r}}^n }\prod_{j=1}^{r} \frac{1}{|(L \cap F_j)^{\circ}|^{ i_{j+1}- i_{j-1}} } d \xi_{\bf r} \\ 
 & = &c^{i_{r}n} \left( \prod_{j=1}^{r} | B_{2}^{n} \cap F_j|^{i_{j+1}- i_{j-1}}\right)^{2}
 \int_{F_{{\bf r}}^{n}}\prod_{j=1}^{r} \frac{1}{ | P_{F_j}L^{\circ}|^{ i_{j+1}-i_{j-1}} } d \xi_{\bf r}\\ 
 &= & c^{i_{r}n} \Phi_{{\bf r}}^{i_{r} n} ( B_{2}^{n} )
 \Psi_{{\bf r}}^{i_{r} n} (B_{2}^{n})\Phi_{{\bf r}}^{-i_{r}n}(L^{\circ}).
 \end{eqnarray*}
 The proof is complete. 
\end{proof}

\noindent The following corollary has been proved in the case $r=1$ in \cite{PP2}.

\begin{corollary}
\noindent Let $ 1 \leq r \leq n-1$ and	$ {\bf r} := ( i_{1}, \cdots, i_{r})$ be an increasing sequence of	integers between $1$ and $n-1$.
Then for every  convex body $L$, 
\begin{equation}
\label{tilde-sym-two}
\Phi_{{\bf r}} (L)  \geq  c \, \Phi_{{\bf r}} (r_{L} B_{2}^{n}) ,
\end{equation}
where $c>0$ is an absolute constant.  
\end{corollary}

A compact set is called {\it centered} if its centroid lies at the origin. 

\begin{proof}
As $\Phi_{{\bf r}}(L)$ is translation-invariant, we may assume that
$L$ is centered.  The Blaschke-Santal\'o inequality implies
$r_L r_{L^{\circ}}\leq 1$, so with \eqref{RS for Psi} and
\eqref{Psi-ineq}, we obtain
$$ \Phi_{{\bf r}} (L) \geq c \frac{ \Phi_{{\bf r} } (B_{2}^{n})
  \Psi_{{\bf r}} (B_{2}^{n}) }{ \Psi_{{\bf r}}(L^{\circ})} \geq c
\frac{ \Phi_{{\bf r} } (B_{2}^{n}) \Psi_{{\bf r}} (B_{2}^{n}) }{
  \Psi_{{\bf r}}( r_{L^{\circ}} B_{2}^{n})} = \frac{c}{ r_{L^{\circ}}}
\Phi_{{\bf r}} ( B_{2}^{n}) \geq c \, \Phi_{{\bf r}} (r_{L} B_{2}^{n}). $$
The proof is complete.
\end{proof}

The next proposition shows that all the quantities $\Phi_{{\bf r}}(L)$
lie between the volume-radius $r_{L}$ and the mean width $W(L)$. For a convex body $L$ in $\mathbb R^{n}$, we write
\begin{equation}\label{min-W}
  W_{L} := \inf_{T\in SL_{n}} W(T L) .
\end{equation}

\begin{proposition}
Let $ 1 \leq r \leq n-1$ and $ {\bf r} := ( i_{1}, \cdots, i_{r})$ be an increasing sequence of integers between $1$ and $n-1$. Then for a convex body $L$ in $\mathbb R^{n}$
	\begin{equation}\label{Bet-Ur}
  		c \, r_L \leq \frac{ \Phi_{{\bf r}}(L) }{ \Phi_{{\bf r}} (B_{2}^{n}) }\leq W_{L} . 
	\end{equation}
\end{proposition}

\begin{proof}
The left-most inequality follows from \eqref{tilde-sym-two}.  Next,
 since $h_{L} (x) = h_{P_{F}L}(x)$ for $x\in F$, we have
\begin{equation}
\label{fub-W}
\int_{G_{n,k}} W(P_{F} L) \, d F = W(L).
\end{equation}
We will prove the right-hand side for the case $r=2$ (the general case follows by induction on $r$).  Using the Urysohn and H\"older inequalities repeatedly, we get
\begin{eqnarray*}
  \Phi_{{\bf r}}(L) 
  & = & \left( \int_{F^n_{\bf r}} \prod_{j=1}^{2} |P_{F_{j}}L|^{-i_{j+1}+i_{j-1}} d\xi_{\bf r} \right)^{-\frac{1}{n i_2}} \\ 
  & = & \left( \int_{G_{n,i_{2}}} \frac{1}{ |P_{F_{2}} L|^{ n-i_{1}}} \int_{G_{F_2,i_1}} \frac{1}{|P_{F_{1}}L|^{i_2}} dF_1 dF_2 \right)^{-\frac{1}{ni_{2}}} \\ 
  &\leq& \left( \int_{G_{n,i_{2}}} \frac{1}{| P_{F_{2}} L|^{ n-i_{1}}} \int_{G_{F_{2}, i_{1}}}
  \frac{1}{|P_{F_{1}}B_{2}^{n}|^{ i_{2}} W(P_{F_{1}}L)^{i_{1} i_{2}}}
  dF_1 dF_2\right)^{-\frac{1}{ni_{2}}} \\ 
  & \leq & \left( \int_{G_{n,i_{2}}} \frac{1}{| P_{F_{2}} L|^{n-i_{1}}} \frac{1}{|B_{2}^{i_{1}}|^{ i_{2}} \left(\int_{G_{F_{2}, i_{1}}} W(P_{F_{1}}L) dF_1 \right)^{i_{1}i_{2}}} dF_2 \right)^{-\frac{1}{ni_{2}}} \\ 
  & = & \left( \int_{G_{n,i_{2}}} \frac{1}{| P_{F_{2}} L|^{ n-i_{1}}} \frac{1}{|B_{2}^{i_{1}}|^{ i_{2}} W(P_{F_{2}}L)^{i_{1}i_{2}}} dF_2 \right)^{-\frac{1}{ni_{2}}} \\ 
  & \leq & \left( \int_{G_{n,i_{2}}} \frac{1}{ | P_{F_{2}} B_{2}^{n}|^{ n-i_{1}} W(P_{F_{2}}L)^{i_{2}(n-i_{1})}} \frac{1}{|B_{2}^{i_{1}}|^{ i_{2}} W(P_{F_{2}}L)^{i_{1}i_{2}}} dF_2
  \right)^{-\frac{1}{ni_{2}}} \\ 
  & = & \left( |B_{2}^{i_{2}}|^{n-i_{1}} |B_{2}^{i_{1}}|^{i_{2}} \right)^{\frac{1}{ ni_{2}}}\left(
  \int_{G_{n,i_{2}}} W(P_{F_{2}}L)^{-n i_{2}} dF_2 \right)^{-\frac{1}{ni_{2}}} \\ 
  & \leq & \Phi_{{\bf r}}(B^n_2) \int_{G_{n,i_{2}}} W(P_{F_{2}}L) dF_2 \\ 
  & = &\Phi_{{\bf r}}(B^n_2) W(L).
\end{eqnarray*}
In the above argument we may replace $L$ by $TL$ with $T\in SL_n$. Since the left-hand side of this inequality remains the same for all $T$ by Theorem \ref{thm-aff}, we may take the infimum over all $T$ on the right-hand side. This completes the proof.
\end{proof}


We conclude this subsection with a discussion of inequalities of isomorphic nature. For convex bodies $L$ in $\mathbb{R}^n$, we define the Banach-Mazur distance to the Euclidean ball $B_2^n$ by $$ d_{BM}(L) := \inf\left\{ ab : a>0,b>0, \frac{1}{b} B_{2}^{n} \subseteq
T(L-L) \subseteq a B_{2}^{n}, T\in GL_{n} \right\}.$$ 
For symmetric convex bodies, this coincides with the standard notion of Banach-Mazur distance (for more information see, e.g., \cite{NTJ89}).

\begin{proposition}\label{prop:iso}
Let $ 1 \leq r \leq n-1$ and $ {\bf r} := ( i_{1}, \cdots, i_{r})$ be an increasing sequence of integers between $1$ and $n-1$. Then for a convex body $L$ in $\mathbb{R}^n$
\begin{equation}\label{Phi-r-upper}
  \Phi_{{\bf r}}(L) \leq c \min\left\{ \sqrt{\frac{n}{i_{r}}},
  \log{(1+ d_{BM}(L))} \right\} \Phi_{{\bf r}} (r_{L} B_{2}^{n}) .
\end{equation}
Moreover, if $L$ is also symmetric, then
\begin{equation}\label{Psi-r-upper}
  \Psi_{{\bf r}}(L) \geq \frac{c}{\min\left\{ \sqrt{\frac{n}{i_{r}}} ,
    \log{(1+ d_{BM}(L))} \right\} }\Psi_{{\bf r}} (r_{L} B_{2}^{n}) .
\end{equation}
\end{proposition}

The proof relies on several different tools. We draw on ideas from Dafnis and the second-named author \cite{DP} to exploit the affine invariance of $\Phi_{{\bf r}}(L)$, $\Psi_{{\bf r}}(L)$ by using appropriately chosen affine images of $L$. To this end, recall the following fundamental theorem which combines work of Figiel--Tomczak-Jaegermann \cite{FTJ}, Lewis \cite{Lewis}, Pisier \cite{Pi82} and Rogers-Shephard \cite{RS2} (see Theorem 1.11.5 on p. 52 in \cite{BGVV}). 


\begin{theorem}\label{thm:ell}
Let $L$ be a centered convex body. Then there exists a linear map $T\in SL_{n}$ such that
  \begin{equation}
    \label{pisier}
    W(T L) \leq c \log\{ 1+ d_{BM} (L)\} \sqrt{n}|L|^{1/n}.
  \end{equation}
\end{theorem}

We will also use recent results on isotropic convex bodies. For background, the reader may consult \cite{BGVV}, we will however recall all facts that we need here. To each convex body $M\subseteq \mathbb{R}^n$ with unit volume, one can associate an ellipsoid $Z_{2}(M)$, called the {\it $L_{2}$-centroid body} of $M$, which is defined by its support function as
$$ h_{Z_{2}(M)} (\theta) := \left( \int_{M} | \langle x, \theta \rangle |^{2} dx \right)^{\frac{1}{2}}. $$ 
The {\it isotropic constant} of $M$ is defined by $L_{M} := r_{Z_{2}(M)}$. We say that $M$ is {\it isotropic} if it is centered and $Z_{2}(M) = L_{M} B_{2}^{n}$.  Fix an isotropic convex body $M$ and a $k$-dimensional subspace $F$. K. Ball \cite{Ball1} proved that
\begin{equation}
\label{Ball}
|M \cap F^{\perp} |^{\frac{1}{ k}} \geq \frac{c}{L_{M}};
\end{equation}
a corresponding inequality for projections,
\begin{equation}
\label{R2-iso}
|P_{F} M| \leq \left(c \frac{n}{k} L_{M}\right)^{k},
\end{equation}
follows immediately from \eqref{Ball} and the Rogers-Shephard inequality \cite{RS2}: 
$$ |P_{F} M| |M\cap F^{\perp}| \leq {n\choose k}. $$

Next, we recall a variant of $\Psi_{[k]}(M)$ studied by Dafnis and the second-named author \cite{DP}. For every $ 1\leq k \leq n-1$ and a compact set $M$ in $\mathbb R^{n}$ with $|M|=1$, we define the quantity
\begin{equation}\label{phi-tilde}
\wtd{\Phi}_{[k]}(M) := \left( \int_{G_{n,k}} |M\cap F^{\perp}|^{n}
d\mu_{G_{n,k}}(F) \right)^{\frac{1}{nk}}.
\end{equation}
In \cite{DP} it is shown that for every convex body $M$ in $\mathbb{R}^n$ of unit volume,
\begin{equation}\label{phi-tilde-bound}
 \frac{c_{1}}{ L_{M}} \leq \wtd{\Phi}_{[k]}(M) \leq \wtd{ \Phi}_{[k]}(D_n) \simeq 1 ,
\end{equation}
where $D_n$ is the Euclidean ball of volume one.

We also invoke B. Klartag's fundamental result on perturbations of
isotropic convex bodies \cite{Kl} having a well-bounded isotropic
constant. 

\begin{theorem}\label{Kl}
Let $M$ be a convex body in $\mathbb R^{n}$. For every $\varepsilon\in (0,1)$ there exists a centered convex body $M_{Kl} \subset \mathbb R^{n}$ and a point $x\in \mathbb R^{n}$ such that
 \begin{equation}
   \label{Kl-1}
   \frac{1}{ 1+\varepsilon} M_{Kl} \subseteq M + x \subseteq (1+\varepsilon) M_{Kl}
 \end{equation}
 and 
 \begin{equation}
   \label{Kl-2}
   L_{M_{Kl}} \leq \frac{c}{ \sqrt{\varepsilon}}.
 \end{equation}
\end{theorem}


We are now ready to complete the proof.

\begin{proof}[Proof of Proposition \ref{prop:iso}]
By homogeneity of the operators $\Phi_{\bf r}$ and $\Psi_{\bf r}$, we can assume that $L$ has unit volume. 

First we will prove the bound \eqref{Phi-r-upper} for $\bf r$-flag
affine quermassintegrals. By translation-invariance of projections, we may
further assume that $L$ is centered. Bounding $\Phi_{{\bf r}}(L)$ by
$W(L)$ according to \eqref{Bet-Ur}, using affine invariance of
$\Phi_{\bf r}$ and reverse Urysohn inequality from Theorem
\ref{thm:ell}, we get
	\begin{equation}
	\label{r-flag-upper-1}
	\Phi_{{\bf r}} (L) \leq c\log( 1+ d_{BM}(L)) \Phi_{{\bf r}}(D_n).
	\end{equation}
For the Euclidean ball $D_n$ of unit volume, for every $F\in G_{n,k}$, we have 
$|P_{F} D_{n}|^{\frac{1}{k}}=| D_{n} \cap F|^{\frac{1}{k}}\simeq \sqrt{\frac{n}{k}}$, 
so
$$ \Phi_{{\bf r}} ( D_{n}) \simeq \left( \prod_{j=1}^{r} \left(\frac{n}{ i_{j}}\right)^{i_{j} (i_{j+1}-i_{j-1}) }\right)^{\frac{1}{2i_{r} n}}. $$ 
The AM/GM inequality implies
$$\left( \prod_{j=1}^{r} \left( \frac{n}{ i_{j}}\right)^{i_{j} (i_{j+1}-i_{j-1}) }\right)^{\frac{1}{ 2i_{r} n}} \leq \sqrt{\frac{n}{i_{r}n} \sum_{j=1}^{r} \frac{i_{j} (i_{j+1}-j_{j-1})}{i_{j}} }\leq \sqrt{ \frac{n}{ i_{r}} }. $$
Thus
\begin{equation}
   \label{est-B}
   \Phi_{{\bf r}} (D_{n}) = \Psi_{{\bf r}}( D_{n}) \simeq \left(
   \prod_{j=1}^{r} \left( \frac{n}{ i_{j}}\right)^{i_{j} (
     i_{j+1}-i_{j-1}) }\right)^{\frac{1}{ 2i_{r} n}} \leq \sqrt{\frac{n}{i_{r}}}.
\end{equation}
Let $K_{1} \subset \mathbb{R}^n$ be a centered convex body and $x\in \mathbb{R}^n$ from the conclusion of Theorem \ref{Kl} corresponding to $\varepsilon=\frac{1}{2}$. Then \eqref{Kl-1} implies $1=|L|^{1/n} \geq \frac{2}{3} | K_{1}|^{1/n}$, while \eqref{Kl-2} implies $L_{K_1} \simeq 1$. Let $ K_{2} := \frac{K_{1}}{ |K_{1}|^{\frac{1}{n}} }$, then $L_{K_2} \simeq L_{K_1} \simeq 1$
and \begin{equation}
\label{phi-r-upper-1-1}
  \Phi_{{\bf r}}(L)= \Phi_{{\bf r}} (L+x) \leq
  \frac{3}{2} \Phi_{{\bf r}}(K_{1}) \leq \frac{9}{4} \Phi_{{\bf r}}(
  K_{2}).
\end{equation}
Affine invariance of $\Phi_{\bf r}$ (Theorem \ref{thm-aff}), allows us to assume that $K_{2}$ is isotropic.
Using \eqref{R2-iso}, \eqref{est-B}, $L_{K_2} \simeq 1$ and \eqref{est-B} one more time, we obtain
\begin{eqnarray*}
  \Phi_{{\bf r}}(K_{2}) 
  &= & \left( \int_{F_{{\bf r}}^{n}} \prod_{j=1}^{r} | P_{F_{j}} K_2|^{-i_{j+1}+i_{j-1}} d\xi_{\bf r} \right)^{-\frac{1}{i_{r} n}} \\ 
  &\leq & \left(\prod_{j=1}^{r} \left( \frac{n}{ i_{j}}\right)^{i_{j}(i_{j+1}-i_{j-1})} \right)^{\frac{1}{i_{r}n}} (cL_{K_{2}})^{\frac{1}{i_{r}n}\sum_{j=1}^{r} i_{j}(i_{j+1}-i_{j-1}) }\\ 
  & \leq & c \, L_{K_{2}} \Phi_{{\bf r}}(D_n)^{2} \\ 
  &\leq & c \, \sqrt{\frac{n}{i_{r}}} \, \Phi_{{\bf r}}(D_n).
\end{eqnarray*}
By \eqref{phi-r-upper-1-1} we have $ \Phi_{{\bf r}} (L) \leq c^{\prime} \sqrt{\frac{n}{i_{r}}} \, \Phi_{{\bf r}} (D_n) $, which together with \eqref{r-flag-upper-1} gives the upper bound \eqref{Phi-r-upper}.

Applying \eqref{RS for Psi} for $L^{\circ}$, \eqref{Phi-r-upper} and the Blaschke-Santal\'{o} inequality $r_L r_{L^{\circ}}\leq 1$, we get
\begin{eqnarray*}
	\Psi_{{\bf r}} (L) 
	&\geq & c\frac{ \Phi_{{\bf r}} (B_{2}^{n}) \Psi_{{\bf r}}(B_{2}^{n}) }{ \Phi_{{\bf r}} ( L^{\circ}) } \\ 
	&\geq & \frac{c}{r_{L^{\circ}}} \frac{1}{ \min\left\{ \log{(1+ d_{BM}(L^{\circ}))}, \sqrt{\frac{n}{i_r}}\right\}} \frac{\Phi_{{\bf r}} (B_{2}^{n}) \Psi_{{\bf r}}(B_{2}^{n}) }{ \Phi_{{\bf r}} ( B_{2}^{n})} \\ 
	& \geq & \frac{c}{ \min\left\{ \log{(1+ d_{BM}(L))}, \sqrt{\frac{n}{i_{r}}}\right\}} \Psi_{{\bf r}} ( r_{L} B_{2}^{n}),
\end{eqnarray*}
where we have also used the identity $d_{BM} (L^{\circ}) =  d_{BM}(L)$ for symmetric convex bodies. This proves \eqref{Psi-r-upper}.
\end{proof}



\section{Flag manifolds and permutations}

In this section, we discuss more general quantities involving
permutations. We investigate the extent to which $SL_{n}$-invariance
properties established by Furstenberg and Tzkoni \cite{FT} carry over
from ellipsoids to compact sets.  In particular, we provide an example
of a convex body for which $SL_{n}$-invariance fails. Nevertheless, we
show that for convex bodies, such quantities cannot be too degenerate
in the sense that they admit uniform upper and lower bounds,
independent of the body. The key ingredient is the notion of
$M$-ellipsoids, introduced by V. Milman \cite{Mil88}.

The next definition is motivated by the work of Furstenberg and Tzkoni
\cite{FT} for ellipsoids.

\begin{definition}
  Let $\Pi_{n}$ be the set of permutations of $\{1,2,\ldots,n\}$ and
  $\omega\in \Pi_{n}$.  For compact sets $L$ in $\mathbb{R}^n$, we
  define the {\it $\omega$-flag quermassintegral} and {\it $\omega$-flag dual
  quermassintegrals} as follows: if $\omega(n) \neq n$, then
\begin{equation}
\label{perm-Psi-def}
\Psi_{\omega} (L) := \left( \int_{F^{n}} \prod_{j=1}^{n-1} | L \cap
F_{j}|^{ \omega(j) - \omega(j+1)+1} d \xi \right)^{\frac{1}{ n(n-
    \omega(n))}}
\end{equation}
and
\begin{equation}
\label{perm-Phi-def-1}
\Phi_{\omega} (L) := \left( \int_{F^{n}} \prod_{j=1}^{n-1} | P_{F_{j}} L|^{- \omega(j) + \omega(j+1)-1} d \xi \right)^{-\frac{1}{n( n- \omega(n))}} .
\end{equation}
When $\omega(n)=n$, we set
\begin{equation}
\label{perm-Psi-def-2}
\Psi_{\omega} (L) := \int_{F^{n}} \prod_{j=1}^{n-1} | L \cap F_{j}|^{
  \omega(j) - \omega(j+1)+1} d \xi
\end{equation}
and 
\begin{equation}
\label{perm-Phi-def-2}
\Phi_{\omega} (L) := \int_{F^{n}} \prod_{j=1}^{n-1} | P_{F_{j}} L|^{-
  \omega(j) + \omega(j+1)-1} d \xi.
\end{equation}
\end{definition}

\noindent Note that
\begin{equation}
\label{perm-00-01}
\sum_{j=1}^{n-1} \left( \omega(j) - \omega(j+1) +1\right) =
n-\omega(n)+ \omega(1)-1
\end{equation}and
\begin{equation}
\label{perm-00-02}
\sum_{j=1}^{n-1} j\left( \omega(j) - \omega(j+1) +1\right) = n(n-\omega(n)).
\end{equation}


Identity \eqref{perm-00-02} guarantees that $\Psi_{\omega}(L)$ and $
\Phi_{\omega}(L)$ are $1$-homogeneous when $\omega(n)\neq n$ and
$0$-homogeneous when $ \omega(n)=n$.

The following fact for $\omega$-flag dual quermassintegrals is from \cite{FT}. For $\omega$-flag quermassintegrals it follows for example by duality.
\begin{theorem}
  \label{thm:FTperms}
  Let ${\cal{E}}$ be an ellipsoid in $\mathbb R^{n}$ and $\omega\in
   \Pi_{n}$.
\begin{equation}
  \label{perm-1}
  \Psi_{\omega} ({\cal{E}}) = \Psi_{\omega} (r_{\cal{E}}B_2^n) \ {\rm
    and} \ \Phi_{\omega} ({\cal{E}}) = \Phi_{\omega}
  (r_{\cal{E}}B_2^n).
\end{equation}
\end{theorem}
\noindent An equivalent formulation of the latter result is that for every
ellipsoid ${\cal{E}}$,
\begin{equation}
\label{perm-ell-1}
 \Psi_{\omega} ({\cal{E}}) = c_{\omega} |{\cal{E}}|^{\frac{1}{n}}, \ \omega(n)
 \neq n \ {\rm and } \ \Psi_{\omega} ({\cal{E}}) = c_{\omega} , \ \omega(n) =
 n,
 \end{equation}
where $c_{\omega} $ is a constant that depends only on $\omega$. An
analogous statement holds for $\Phi_{\omega}(\cal{E})$
(cf. \eqref{prop:BS_flag}).

The operators $\Psi_{\omega}$ and $\Phi_{\omega}$ are
generalizations of $\Psi_{{\bf r}}$ and $\Phi_{{\bf r}}$. 
Indeed, let $1\leq r \leq n-1$, $1\leq i_{1}< i_{2} < \cdots < i_{r}\leq n-1$ and
$ {\bf r}:= (i_{1}, \cdots , i_{r}). $ Define $\omega $ by $\omega (1) = n-i_{1}+1$, and  $\omega(t+1)= \omega(t) +1$ , for $ t\neq i_{j}$, $1\leq j \leq r$, and $\omega( i_{j}+1)= \omega(i_{j}) +1 - i_{j+1}+ i_{j-1}$ for $ 1\leq j \leq r$.  Then $\omega \in \Pi_n$ with $
\omega(i_{j} ) - \omega(i_{j}+1) +1 = i_{j+1}- i_{j-1}$ for $1\leq
j\leq r$ and $ \omega ( t) - \omega( t+1) +1 = 0 , \ t\neq i_{j}$, for
$1\leq j \leq r.$ Since $\omega(n) = n-i_{r}$, for a compact set $L$ in $\mathbb R^{n}$ we have
\begin{eqnarray*}
\Psi_{\omega} (L) 
&= & \left( \int_{F^{n} } \prod_{j=1}^{n-1} |L \cap F_{j}|^{\omega(j)-\omega(j+1)+1} d \xi \right)^{\frac{1}{ ni _{r}}}\\ 
&=& \left( \int_{F^{n} } \prod_{j=1}^{r} |L \cap F_{i_{j}}|^{i_{j+1}-i_{j-1}} d \xi \right)^{\frac{1}{ ni _{r}} } \\
&=& \Psi_{{\bf r}}(L),
\end{eqnarray*}
where, in the last equality, we have used \eqref{inv-proj}. Correspondingly, 
we also have $ \Phi_{\omega}(L) = \Phi_{{\bf r}} (L)$.  In particular, for this permutation $\omega$, $\Psi_{\omega}(L)$ is $SL_{n}$-invariant and $\Phi_{\omega}(L)$ is affine invariant. 
As a particular case of the preceding discussion, let $r=1$, $i_{1}= k$,
$1\leq k \leq n$ and let $\omega(1)= n-k+1$ and $\omega(t+1)=
\omega(t)+1$ for $t\neq k$ and $ \omega(k+1) = \omega(k) - n +1$. Then
$\Phi_{\omega}(L) = \Phi_{[k]}(L).$

Given that $\Psi_{{\bf r}}(L)$ and $\Phi_{{\bf r}}(L)$ enjoy
invariance properties and arise as permutations, it is natural to
investigate the extent to which the invariance from Theorem
\ref{thm:FTperms} carries over to compact sets. We do not have a
complete answer.  However, there are cases outside of those considered
above where the invariance holds and also counter-examples where it
fails as the next two examples show.

\begin{example} Let $n\geq 3$. Define $\omega$ by
$\omega(1)=2$, $\omega(2)=1$ and $\omega(t)=t$ for all $ 3\leq t \leq n$. Then for every symmetric compact set $L$ in $\mathbb R^{n}$,
$$ \Psi_{\omega}(L) = \frac{4}{\pi} . $$
In particular, $\Psi_{\omega}(L)$ is $SL_{n}$-invariant.  

Note that our choice of the permutation $\omega$ satisfies 
$$ \omega(1)-\omega(2)+1=2, \quad \omega(2)-\omega(3)+1=-1, \quad
  \omega(j)-\omega(j+1)+1=0 \text{ for } 3\leq j \leq n-1,$$ 
or equivalently
$$ \omega(2)=\omega(1)-1, \quad \omega(3)=\omega(1)+1, \quad
  \omega(j+1)=\omega(1)+(j-1) \text{ for } 3\leq j \leq n-1. $$
  Since $1\leq \omega(j)\leq n$ for all $j$, it follows that
  $\omega(1)=2$. Hence $\omega(1)=2, \omega(2)=1, \omega(j)=j$ for
  $3\leq j \leq n$ is the unique permutation with these
  properties. For an $k$-dimensional subspace $F_k$ of $\mathbb{R}^n$, denote by $S_{F_k}$ the unit sphere in $F_k$. Now, using \eqref{inv-proj}, we compute
\begin{align*}
\Psi_{\omega} (L) 
&= \int_{F^{n}} \prod_{j=1}^{n-1} |L \cap F_{j}|^{\omega(j)-\omega(j+1)+1} d\xi \\ 
&= \int_{F^{n}} |L \cap F_1|^2 |L \cap F_2|^{-1} d\xi \\ 
&= \int_{G_{n,2}}  |L \cap F_2|^{-1} \int_{G_{F_2,1}} |(L \cap F_2) \cap F_1|^2 dF_1 dF_2 \\
&= \int_{G_{n,2}}  |L \cap F_2|^{-1} \int_{S_{F_2}} (2 \rho_{L \cap F_2}(\theta))^2 d\sigma(\theta) dF_2 \\ 
&= \int_{G_{n,2}}  |L \cap F_2|^{-1} \frac{4}{|S_{F_2}|} \int_{S_{F_2}} \rho^2_{L \cap F_2}(\theta) d\theta \, dF_2 \\ 
&= \frac{4}{\pi} \int_{G_{n,2}}  |L \cap F_2|^{-1}  |L \cap F_2| dF_2 \\ 
&= \frac{4}{\pi} .
\end{align*}
\end{example}
\qed

When $n=3$, for permutations $\omega$ with $\omega(3)=3$,
$\Psi_{\omega}(L)$ are absolute constants. Moreover, the discussion
following Theorem \ref{thm:FTperms} shows that for $3$ of the
remaining $4$ permutations $\omega$ in $\Pi_3$,
$\Psi_{\omega}(L)=\Psi_{\bf r}(L)$. Altogether, for $n=3$, for 5 out
of 6 permutations $\omega$, $\Psi_{\omega}(L)$ are
$SL_n$-invariant. The next example shows that for the remaining
permutation, the invariance does not carry over for all convex bodies.

\begin{example} Let $\omega \in \Pi_3$ with $\omega(1)=1, \, \omega(2)=3$ and $\omega(3)=2$. We claim that for a centered cube $Q:=[-1,1]^3$ and the diagonal matrix $D=\mathop{\rm diag}(1,2,1/2)$,  $\Phi_{\omega}(DQ)>\Phi_{\omega}(Q)$. Since $D \in SL_3$, this shows that the operator $\Phi_{\omega}$ is not invariant under volume preserving transformations.

To show this, we first note that for any convex body $L \subset \mathbb R^{3}$
\begin{equation}
\label{ex-1}
\Phi^{-3}_{\omega}(L) = \int_{S^{2} } \frac{ W( P_{\phi^{\perp}}L)}{ h_{\Pi L}^{2} (\phi) } d\sigma (\phi).
\end{equation}
Recall that for $\theta \in S^{n-1}$, $h_Q(\theta)=\sum_{i=1}^n |\theta_i|$ and for $g\in GL_n$, $h_{gL}(\theta) = h_L(g^{t}\theta)$. 
We will also use the following facts about projection bodies (see e.g., Gardner). 
The projection body of a cube is again a cube, $\Pi Q = 2 Q$ and for $g\in GL_n$, 
\begin{equation}\label{ex-3}
\Pi (g L) = | {\rm det} g | \, g^{-t} \, \Pi L.
\end{equation}

Let $ A=[a_{1} \ a_{2} \ a_{3} ]  \in SL_3$ with columns $a_i$. Fix $ \phi\in S^{2}$. Let $U \in O_3$ be given in column form by $ U = [ u \ v \ \phi ] $. Since $ U $ is orthogonal, $ U^{t} \phi = e_{3} $ and $ U^{t} \phi^{\perp} = {\rm span}\{ e_{1}, e_{2} \} = \mathbb R^{2} $. 
Then
$$ W( P_{\phi^{\perp}} A Q)= \int_{S_{\phi^{\perp}}} h_{A Q } ( \theta) d\sigma (\theta ) = \int_{S^{1}} h_{AQ} ( U\theta ) d \sigma (\theta ) = 
 \int_{S^{1}} h_{Q} ( A^{t} U\theta ) d \sigma ( \theta ).$$ 
Thus denoting by $P$ the orthogonal projection onto $\mathbb R^{2}$, we have
$$ W( P_{\phi^{\perp}} A Q) 
= \sum_{i=1}^{3} \int_{S^{1}} | \langle \theta , U^{t} A e_{i} \rangle | d\sigma( \theta )
= \sum_{i=1}^{3} \int_{S^{1}} | \langle \theta , P U^{t} A e_{i} \rangle | d\sigma( \theta ) 
= \frac{2}{\pi} \sum_{i=1}^{3} \| P U^{t} A e_{i} \|_{2} .$$
We have that $ A e_{i} = a_{i}$, $ U^{t} a_{i} = \left( \langle u,
a_{i} \rangle , \langle v, a_{i} \rangle, \langle \phi, a_{i} \rangle\right)^{t} $ and
$$ \|  P U^{t} A e_{i} \|_{2}^{2} = \| U^{t} A e_{i} \|_{2}^{2} - \| (I-P) U^{t} A e_{i} \|_{2}^{2} = \|a_{i}\|_{2}^{2} - \langle \phi, a_{i} \rangle^{2} . $$
Therefore,
\begin{equation}
\label{ex-4}
 W( P_{\phi^{\perp}} A Q)= \frac{2}{\pi}  \sum_{i=1}^{3} \sqrt{ \|a_{i}\|_{2}^{2} -
   \langle \phi, a_{i} \rangle^{2} } .
\end{equation}
Moreover, $ h_{\Pi (A Q) } ( \phi) = h_{\Pi Q} ( A^{-1} \phi) = 2
\sum_{i=1}^{3} | \langle A^{-1} \phi, e_{i} \rangle | $. Thus
\begin{equation}
\label{ex-5}
\Phi^{-3}_{\omega}(AQ)
= \frac{1}{2\pi} \int_{S^{2}} \frac{ \sum_{i=1}^{3} \sqrt{ \|a_{i}\|_{2}^{2} - \langle \phi, a_{i} \rangle^{2} }}{\left(\sum_{j=1}^{3} | \langle A^{-1} \phi, e_{j} \rangle \right)^{2} } d\sigma ( \phi) .
\end{equation}
Set $ A:= {\rm diag} ( d_{1}, d_{2}, d_{3})$ with $\prod_{i=1}^{3} d_{i}= 1$ and $ d_{i} >0$. Then the quantity
\begin{equation}
\label{ex-6} {\cal{A}}( d_{1}, d_{2}, d_{3}) := \int_{S^{2}} \frac{\sum_{i=1}^{3} d_{i} \sqrt{ 1- \phi_{i}^{2} }}{  \left(\sum_{j=1}^{3}\frac{|\phi_{j}|}{ d_{j}}\right)^{2} } d\sigma ( \phi) 
\end{equation}
is not constant. Indeed, using MATLAB for example, one can verify that
${\cal{A}}(1,2,1/2)<{\cal{A}}(1,1,1)$.
\end{example}
\qed

In the case of convex bodies, the quantities $\Psi_{\omega}(K)$,
$\Phi_{\omega}(K)$ are uniformly bounded.  We will use the following
well-known consequence of the celebrated ``existence of M-ellipsoids" by
V. Milman \cite{Mil88}.



\begin{theorem}
  Let $K$ be a symmetric convex body in $\mathbb R^{n}$. Then there
  exists an ellipsoid ${\cal{E}}$ such that $|{\cal{E}}|^{1/n}
  \leq e^{c} |K|^{1/n}$ and for every $F\in G_{n,k}$,
  \begin{equation}
    \label{M-position-2}
    | P_{F} {\cal{E}}|  \leq  | P_{F} K | \leq e^{cn} | P_{F} {\cal{E}}|
  \end{equation}and 
  \begin{equation}
    \label{M-position-3}
    |  {\cal{E}} \cap F|  \leq  | K \cap F | \leq e^{cn} | {\cal{E}} \cap F|,
  \end{equation}
  where $c>0$ is an absolute constant.
\end{theorem}

\begin{corollary}
  \label{M-FT}
  Let $\omega\in \Pi_{n}$ such that $\omega(n)\neq n$. Let $
  \delta_{\omega} (j) := \omega(j) - \omega(j+1) + 1$. Set
  $$I_{\omega} := \{ j\leq n : \delta_{\omega} ( j) \geq 0 \} \ {\rm and }
  \ \Delta(\omega):=\frac{\min\{ \sum_{j\in I_{\omega} }
    \delta_{\omega}(j) , \sum_{j\in I_{\omega}^{c} }
    |\delta_{\omega}(j)|\} }{n - \omega(n)} + 1 . $$ We have that
\begin{equation}
\label{M-position-6}
 e^{ -c \Delta( \omega) } c_{\omega} | K|^{\frac{1}{n}} \leq
 \Psi_{\omega} (K) \leq e^{ c \Delta( \omega) } c_{\omega} |
 K|^{\frac{1}{n}}
\end{equation}
and
\begin{equation}
\label{M-position-5}
e^{ - c \Delta( \omega)} c_{\omega} | K|^{\frac{1}{n}} \leq
\Phi_{\omega} (K) \leq e^{ c \Delta( \omega)} c_{\omega} |
K|^{\frac{1}{n}}
\end{equation}
where $c>0$ is an absolute constant.
\end{corollary}

\begin{proof}
Set $ \Delta_{+}(\omega) := \frac{ \sum_{j\in I_{\omega} }
  \delta_{\omega}(j) }{n - \omega(n)}$ and $ \Delta_{-} := ( \omega)
\frac{ \sum_{j\in I_{\omega}^{c}} |\delta_{\omega}(j) |}{n - \omega(n)}$. Using \eqref{M-position-3} and \eqref{perm-00-01}, we have
\begin{eqnarray*}
  \Psi_{\omega} (K) &:=& \left( \int_{F_{n}} \prod_{j=1}^{n-1} | K
  \cap F_{j}|^{ \omega(j) - \omega(j+1)+1} d \xi \right)^{\frac{1}{
      n(n- \omega(n))}}\\ &\leq & \left( \int_{F_{n}} e^{cn \sum_{j\in
      I_{\omega}} \omega(j) - \omega(j+1) +1} \prod_{j=1}^{n-1} |
       {\cal{E}} \cap F_{j}|^{ \omega(j) - \omega(j+1)+1} d \xi
       \right)^{\frac{1}{ n(n- \omega(n))}} \\ &\leq & e^{
         c\Delta_{+}(\omega) } \left( \int_{F_{n}} \prod_{j=1}^{n-1} |
             {\cal{E}} \cap F_{j}|^{ \omega(j) - \omega(j+1)+1} d \xi
             \right)^{\frac{1}{ n(n- \omega(n))}}\\ &=&e^{
               c\Delta_{+}(\omega) }c_{\omega} |
                   {\cal{E}}|^{\frac{1}{n}} \\ &\leq & e^{ c (
                     \Delta_{+}( \omega) + 1 )}c_{\omega} | K |^{\frac{1}{ n}}.
\end{eqnarray*}
  
One can verify that a similar inequality with the quantity $
\Delta_{-} (\omega)$ holds as well, which leads to the right-hand side
in \eqref{M-position-6}. The proof of the other inequalities is
identical and hence is omitted.
\end{proof}

\begin{remark} A similar proposition (with the same proof) holds for the case
$\omega(n)=n$. Moreover, using Pisier's regular M-position (see
  \cite{Pis89}) one can get more precise estimates.
  \end{remark}

\section{Functional forms}
\label{section:functions}

In this section we derive functional forms of some of the previous
geometric inequalities. Note that the proofs of these functional
inequalities do not depend on the geometric inequalities. They are
much more general. The invariance of functional inequalities on flag
manifolds can be proved directly using the structure theory of
semi-simple Lie groups as was done in our previous work.

\subsection{Functional forms of dual ${\bf r}$-flag quermassintegrals. }

Let $ f $ be a bounded integrable function on $\mathbb{R}^{n}$. We denote by $I(f)$ the functional form of the dual ${\bf r}$-flag quermassintegral
$$ I(f) := \int_{ F_{\bf r}^n } \prod_{j=1}^{r} \| f |_{F_j}\|^{i_{j+1}-i_{j-1}} d\xi_{\bf r} . $$

\begin{theorem}\label{th_fi}
For  every $ g \in SL_{n}$, $ I( g\cdot f) = I( f)$.
\end{theorem}  

\begin{proof}
Starting with the left-hand side, $I( g\cdot f)$, we do a global change of variables \eqref{eq_gcv} on the flag manifold:
\begin{align*}
I( g\cdot f) 
&= \int_{ F_{\bf r}^n } \prod_{j=1}^{r} \| g\cdot f |_{F_j}\|^{i_{j+1}-i_{j-1}} d\xi_{\bf r} \\
&= \int_{ F_{\bf r}^n } \prod_{j=1}^{r} \| g\cdot f |_{g\cdot F_j}\|^{i_{j+1}-i_{j-1}} \sigma_{F_{{\bf r}}^{n} }(g, \xi) d\xi_{\bf r}, \\ 
\end{align*}
where by \eqref{F-S-02} $\sigma_{F_{{\bf r}}^{n} }(g, \xi) = \sigma_{i_{1}}^{-i_{2}}(g, F_1) \sigma_{i_{2}}^{i_{1}-i_{3}}(g, F_2) \cdots \sigma_{i_{r}}^{i_{r-1}-n}(g, F_r) $.
Now we do $r$ local changes of variables \eqref{eq_lcv} on each nested subspace $F_j$ in the product. For each $1\leq j \leq r$, we thus have 
$$ \| g\cdot f|_{g\cdot F_j}\|=\|f|_{F_j}\| \, \sigma_{i_j}(g,F_j).  $$
For the product under the integral, we obtain: 
$$ \prod_{j=1}^{r} \| g\cdot f |_{g\cdot F_j}\|^{i_{j+1}-i_{j-1}} = \prod_{j=1}^{r}  \left( \|f|_{F_j}\| \, \sigma_{i_j}(g,F_j) \right)^{i_{j+1}-i_{j-1}} = \prod_{j=1}^{r} \|f|_{F_j}\|^{i_{j+1}-i_{j-1}} \sigma_{F_{{\bf r}}^{n} }^{-1}(g, \xi). $$
\end{proof}
It is not hard to generalize this result in several ways as was done in \cite{DPP} for functional forms of dual quermassintegrals. Instead of taking $L_1(F_j)$ norms one can take $L_{p_j}(F_j)$ norms and replace the powers $i_{j+1}-i_{j-1}$ by $\alpha_j$. As long as $\frac{\alpha_j}{p_j}=i_{j+1}-i_{j-1}$ and the integrals exist, the conclusion of the Theorem \ref{th_fi} will hold. Theorem \ref{th_fi} also generalizes to a product of $m$ functions. This allows to replace $ \| f |_{F_j}\|^{i_{j+1}-i_{j-1}} $ by $ \prod_{i=1}^m \|  f_i |_{E_j} \|^{\alpha_{i,j}}_{p_{i,j}} $. For the Theorem \ref{th_fi} to hold in this case, we have to require $\sum_{i=1}^m \frac{\alpha_{i,j}}{p_{i,j}} = i_{j+1}-i_{j-1}$. Another way to generalize functional forms of dual quermassintegrals is to replace 
$$\| f |_{F_j}\|^{i_{j+1}-i_{j-1}} \quad \text{ by } \quad \frac{ \| f |_{F_j} \|^{\alpha_j}_{p_j}}{{\| g |_{F_j}\|^{\beta_j}_{q_j}}} \quad \text{ with } \quad \frac{\alpha_j}{p_j} - \frac{\beta_j}{q_j} = i_{j+1}-i_{j-1}, $$ 
to ensure they remain invariant under volume preserving transformations. Letting $\beta_j \to \infty$ modifies the integrand to $\frac{ \| f |_{F_j} \|^{\alpha_j}_{p_j}}{{\| f |_{F_j}\|^{\beta_j}_{\infty}}}$ and the condition on the powers and norms to $\frac{\alpha_j}{p_j} = i_{j+1}-i_{j-1} $. Note that in this case the invariance holds for arbitrary powers $\beta_j$. As a particular case this proves invariance under volume preserving transformations of the integrand appearing in the next theorem. One can also take the quotient of products of functions, replacing  
$$\| f |_{F_j}\|^{i_{j+1}-i_{j-1}} \quad \text{ by } \quad \frac{ \prod_{i=1}^m \|  f_i |_{F_j} \|^{\alpha_{i,j}}_{p_{i,j}} }{ \prod_{l=1}^{m^{\prime}} \|  g_l |_{F_j} \|^{\beta_{l,j}}_{q_{l,j}} } \quad \text{ with } \quad \sum_{i=1}^m \frac{\alpha_{i,j}}{p_{i,j}} - \sum_{l=1}^{m^{\prime}} \frac{\beta_{l,j}}{q_{l,j}}=i_{j+1}-i_{j-1}. $$ 
Here again we can let $q_{l,j} \rightarrow \infty$, obtaining the corresponding generalization with no restrictions on $\beta_{l,j}$.

\begin{theorem}
  Let $f$ be a non-negative bounded integrable function on $\mathbb
  R^n$, then
$$ \int_{F_{{\bf r}}^n} \prod_{j=1}^r
  \frac{\|f|_{F_j}\|^{i_{j+1}-i_{j-1}}_1}{\|f|_{F_j}\|^{i_{j+1}-i_{j}}_{\infty}}
  d \xi_{\bf r} \leq \prod_{j=1}^r
  \frac{\omega_{i_j}^{i_{j+1}}}{\omega_{i_{j+1}}^{i_j}}
  \|f\|_1^{i_r}. $$
\end{theorem} 

\begin{proof}
The result follows by iteration of an inequality on $G_{n,k}$ for one
function from our previous work \cite{DPP}:
\begin{equation}
  \label{eqn:DPP}
  \int_{G_{n,k}}\frac{\|f|_{E}\|^{n}_1}{\|f|_{E}\|^{n-k}_{\infty}} dE
  \leq \frac{\omega_k^{n}}{\omega_n^k}\|f\|^k_1 .
\end{equation}
\noindent
Applying the latter inequality repeatedly, we get
\begin{align*}
  \int_{F_{{\bf r}}^n} &\prod_{j=1}^r
  \frac{\|f|_{F_j}\|^{i_{j+1}-i_{j-1}}_1}{\|f|_{F_j}\|^{i_{j+1}-i_{j}}_{\infty}}
  d \xi_{\bf r} \\ &= \int_{G_{n,i_r}} \cdots \int_{G_{F_3,i_2}}
  \prod_{j=2}^r
  \frac{\|f|_{F_j}\|^{i_{j+1}-i_{j-1}}_1}{\|f|_{F_j}\|^{i_{j+1}-i_{j}}_{\infty}}
  \int_{G_{F_2,i_1}}
  \frac{\|f|_{F_1}\|^{i_{2}}_1}{\|f|_{F_1}\|^{i_{2}-i_{1}}_{\infty}}
  dF_1 dF_2 \ldots dF_r\\ &\leq
  \frac{\omega_{i_1}^{i_2}}{\omega^{i_1}_{i_2}} \int_{G_{n,i_r}}
  \cdots \int_{G_{F_3,i_2}} \prod_{j=2}^r
  \frac{\|f|_{F_j}\|^{i_{j+1}-i_{j-1}}_1}{\|f|_{F_j}\|^{i_{j+1}-i_{j}}_{\infty}}
  \|f|_{F_2}\|^{i_{1}}_1 dF_2 \ldots dF_r \\ &=
  \frac{\omega_{i_1}^{i_2}}{\omega^{i_1}_{i_2}} \int_{G_{n,i_r}}
  \cdots \int_{G_{F_4,i_3}} \prod_{j=3}^r
  \frac{\|f|_{F_j}\|^{i_{j+1}-i_{j-1}}_1}{\|f|_{F_j}\|^{i_{j+1}-i_{j}}_{\infty}}
  \int_{G_{F_3,i_2}}
  \frac{\|f|_{F_2}\|^{i_{3}}_1}{\|f|_{F_2}\|^{i_{3}-i_{2}}_{\infty}}
  dF_2 dF_3 \ldots \ldots dF_r \\ &\leq
  \frac{\omega_{i_1}^{i_2}}{\omega^{i_1}_{i_2}}
  \frac{\omega_{i_2}^{i_3}}{\omega^{i_2}_{i_3}} \int_{G_{n,i_r}}
  \cdots \int_{G_{F_4,i_3}} \prod_{j=3}^r
  \frac{\|f|_{F_j\|^{i_{j+1}-i_{j-1}}_1}}{\|f|_{F_j}\|^{i_{j+1}-i_{j}}_{\infty}}
  \|f|_{F_3}\|^{i_{2}}_1 dF_3 \ldots dF_r\\ &=\cdots \\ &\leq \prod_{j=1}^r
  \frac{\omega_{i_j}^{i_{j+1}}}{\omega_{i_{j+1}}^{i_j}} \|f\|_1^{i_r}.
\end{align*}
\end{proof}
  
In \cite{DPP}, more general versions of \eqref{eqn:DPP} are proved
with multiple functions and different powers.  These also carry over
to extremal inequalities on flag manifolds by mimicking the previous
proof. As a sample, we mention just one statement. Let $1 \leq q \leq
i_1$ and let $f_1, \ldots, f_q$ be non-negative bounded integrable
functions on $\mathbb R^n$, then
\begin{equation}
  \label{eqn:ext}
\int_{F_{{\bf r }}^n} \prod_{k=1}^q \prod_{j=2}^r
\frac{\|f_k|_{F_j}\|^{\frac{i_{j+1}-i_{j}}{i_j}}_1}{\|f_k|_{F_j}\|^{\frac{i_{j+1}-i_{j}}{i_j}}_{\infty}}
\frac{\norm{f_k|{F_1}}_1^{\frac{i_2}{i_1}}}{\norm{f_k|_{F_1}}^{\frac{i_2-i_1}{i_1}}_{\infty}}
d \xi_{\bf r} \leq \left( \prod_{j=1}^r
\frac{\omega_{i_j}^{\frac{i_{j+1}}{i_j}}}{\omega_{i_{j+1}}} \right)^q
\prod_{k=1}^q \|f_k\|_1.
\end{equation}

\subsection{Functional forms of the ${\bf r}$-flag quermassintegrals. }

\label{sub:func}

In this subsection we will extend the notions of ${\bf r}$-flag
quermassintegrals to functions. In particular, this will lead to
functional versions of affine quermassintegrals. This is motivated by
recent work of Bobkov, Colesanti and Fragal\'{a} \cite{BCF} and
V. Milman and Rotem \cite{MR}. The latter authors proposed and studied
a notion of quermassintegrals for $\log$-concave or even quasi-concave
functions, which we now recall.

\begin{definition}
  Suppose that $f:\mathbb{R}^n\rightarrow [0,\infty)$ is
    upper-semicontinuous and quasi-concave. For $ 1\leq i \leq n$, let
  \begin{equation}
    \label{quer-functions-1}
    V_k (f) := \int_{0}^{\infty} V_k( \{ f\geq t\}) d t .
  \end{equation}
\end{definition}

The above definition is consistent with the notion of projection of a
function onto a subspace as introduced by Klartag and Milman in
\cite{KlM}. Namely, let $f$ be an non-negative function $\mathbb
R^{n}\rightarrow [0,\infty]$ and $F\in G_{n,k}$. Define the orthogonal
projection of $f$ onto $F$ as the function $P_{F} f: F\rightarrow
[0,\infty]$ given by
\begin{equation}
\label{proj-def}
(P_{F} f) ( z) := \sup_{y\in
  F^{\perp}} f ( z+y) .
\end{equation}
Note that if $K$ is compact and $ f:= {\bf 1}_{K}$ then $P_{F}f :=
{\bf 1}_{P_{F}(K)}$. Moreover, from the definition, one has
\begin{equation}
\label{proj-f-1}
\{ z\in F : (P_{F}f)(z) >t \} = P_{F} ( \{ x\in \mathbb R^{n}: f(x) >t\}).
\end{equation}

Assume now that $f:=\mathbb R^{n} \rightarrow [0,\infty)$ and that for
  each $t>0$, the set $\{x\in \mathbb{R}^n:f(x)\geq t\}$ is
  compact. For $1\leq k \leq n-1$, we define the affine
  quermassintegral of $f$ by
\begin{equation}
  \label{aff-quer-f}
  \Phi_{[k]} (f) := \int_{0}^{\infty} \Phi_{[k]} ( \{ f\geq t\}) dt =
  \int_{0}^{\infty} \left(\int_{G_{n,k}} | \{ P_{F} f\geq t\} |^{-n} d
  F \right)^{\frac{1}{ nk}} dt .
\end{equation}
For $ 1\leq i_{1} < i_{2} < \cdots < i_{r} = n-1$, ${\bf r}:= (i_{1},
\cdots , i_{r})$ we define the ${\bf r}$-flag quermassintegrals of $f$
by
\begin{equation}
\label{aff-flag-quer-f}
\Phi_{{\bf r}} (f) := \int_{0}^{\infty} \Phi_{{\bf r}} ( \{ f\geq t\}) dt.
\end{equation}

For comparison, we recall that for every $f:\mathbb R^{n} \rightarrow
[0,\infty] $,
\begin{equation}
  \label{L1}
  \int_{\mathbb R^{n}} f (x) d x = \int_{0}^{\infty} | \{ f\geq t \}|
  dt.
\end{equation} For $\lambda\in \mathbb R\setminus \{0\}$ and $f$ as above, we write
\begin{equation}
\label{f-lambda}
f_{(\lambda)} :\mathbb R^{n} \rightarrow [0,\infty] , \ {\rm as}
\ f_{(\lambda)} ( x) := f \left(\frac{ x}{ \lambda}\right) ,
\end{equation}
and if $T\in GL_{n}$, 
\begin{equation}
\label{f-T}
f\circ T :\mathbb R^{n} \rightarrow [0,\infty] , \ {\rm as} \ f\circ T ( x) := f (T^{-1} x) .
\end{equation}
\noindent Note that if $f:= {\bf 1}_{K}$, then
$$ f_{(\lambda)} (x) = {\bf 1}_{\lambda K} (x) \ {\rm and} \ f\circ T (x) = {\bf 1}_{TK} (x)  . $$ 
Let $ f:\mathbb R^{n} \rightarrow [0, \infty]$, $\lambda >0$ and $
 T\in GL_{n}$. Then
\begin{equation}
\label{b-f-lambda}
\{ f\circ T \geq t \} = T\left( \{ f\geq t \} \right) \ {\rm and }
\ \{ f_{(\lambda)} \geq t \} = \lambda \{ f \geq t \} .
\end{equation}
Then \eqref{b-f-lambda}, the $1$-homogenuity of the ${\bf r}$-flag
quermassintegrals for sets as well as the affine invariance of these
quantities imply the following.

\begin{theorem}
 Let $f:\mathbb R^{n}\rightarrow [0,\infty]$, $1\leq i_{1}<\cdots <
 i_{r}\leq n-1$ and ${\bf r} := (i_{1}, \cdots , i_{r})$. Let
 $\lambda>0$ and $T$ be an affine volume-preserving map. Then
\begin{equation}
\Phi_{{\bf r}} (f_{(\lambda)} ) = \lambda \Phi_{{\bf r}}(f) \ {\rm and} \ \Phi_{{\bf r}} (f\circ T) = \Phi_{{\bf r}}(f) . 
\end{equation}
\end{theorem}

Recall that the symmetric decreasing rearrangement of a function $f$
which is integrable (or vanishes at infinity). For a set $A\subseteq
\mathbb R^{n}$ with finite volume, the decreasing rearrangement
$A^{\ast}$ is defined as
$$ A^{\ast} := r_{A} B_{2}^{n}, $$ where $r_A$ is the volume-radius of
$A$. The symmetric decreasing rearrangement $f^{\ast}$ of $f$ is
defined as the radial function $f^{\ast}$ such that
$$\{f\geq t\}^{\ast} = \{ f^{\ast} \geq t \} , \ \forall t>0 . $$
Thus,
\begin{equation}
  \label{basic-f-1-0}
  r_{\{ f\geq t\} } B_{2}^{n} = \{ f^{\ast} \geq t \} . 
\end{equation}

Using \eqref{basic-f-1-0}, \eqref{aff-flag-quer-f} and
\eqref{tilde-sym-two}, we have the following for all non-negative
quasi-concave functions $f$ on $\mathbb R^{n}$:
$$ \Phi_{{\bf r}} (f) = \int_{0}^{\infty} \Phi_{{\bf r}}(\{ f\geq t\})
d t \geq c \int_{0}^{\infty} \Phi_{{\bf r}}(r_{\{ f\geq t\}}
B_{2}^{n}) d t = $$
$$ c \int_{0}^{\infty} \Phi_{{\bf r}}(\{ f^{\ast} \geq t \} ) d t=
\Phi_{{\bf r}} (f^{\ast}) . $$

Let $f$ be a non-negative quasi-concave function on $\mathbb
R^{n}$. We define
$$ d_{BM} ( f) := \sup_{t>0}d_{BM} ( \{ f\geq t\} ). $$ The results of
\S 3 lead to the following double-sided inequality for $\Phi_{[{\bf
      r}]}(f)$:
\begin{theorem}
  Let $f$ be a non-negative quasi-concave function on $\mathbb R^{n}$,
  $1\leq i_{1}<\cdots <i_{r} $ and let ${\bf r}:= (i_{1}, \cdots ,
  i_{r})$. Then
  \begin{equation}
    \label{ineq-funct-quer}
    c \Phi_{{\bf r}} (f^{\ast} ) \leq \Phi_{{\bf r}} ( f)  \leq c^{\prime} \min\left\{ \log\{ 1+d_{BM}(f)\}, \sqrt{\frac{n}{ i_{r}}} \right\}    \Phi_{{\bf r}} (f^{\ast} ) .
  \end{equation}
\end{theorem}


\footnotesize \bibliographystyle{amsplain}

\end{document}